\documentclass[a4paper,11pt]{article}

\usepackage[top=3.5cm, bottom=3.5cm, left=2.5cm, right=2.5cm]{geometry}

\usepackage[utf8]{inputenc}
\usepackage[T1]{fontenc}
\usepackage[french,english]{babel} 
\usepackage{enumerate}

\usepackage{xcolor}
\usepackage[]{pdfpages}

\usepackage{titling}
\usepackage[colorlinks=false,linkcolor=blue]{hyperref}
\usepackage{csquotes}
\usepackage[nottoc, notlof, notlot, numbib]{tocbibind}
\usepackage[style=alphabetic,giveninits,doi=false,isbn=false,url=false,eprint=false,clearlang=true,sorting=nyt,maxbibnames=99]{biblatex}
\DeclareSourcemap{
      \maps[datatype=bibtex]{
      \map[overwrite=false]{
       \step[fieldsource=note]
       \step[fieldset=addendum, origfieldval, final]
       \step[fieldset=note, null]
       }
   }
   }

\usepackage{amsthm,amsfonts,amsmath,amssymb}
\usepackage{mathtools}
\usepackage{thmtools}
\usepackage{bbm}
\usepackage{mathrsfs}
\usepackage{float}
\usepackage{stmaryrd}
\usepackage{setspace}
\usepackage{array}
\usepackage{tabularx}
\usepackage{ltablex} 

\usepackage{todonotes}
\mathtoolsset{showonlyrefs}

\newcommand{\A}{\mathcal{A}}

\newcommand{\C}{\mathcal{C}}

\newcommand{\R}{\mathbb{R}}

\newcommand{\T}{\mathcal{T}} 
 
\newcommand{\1}{\mathbbm{1}}
\newcommand{\E}{\mathbb{E}}
\newcommand{\ps}{\mathcal{P}} 

\renewcommand{\d}{\mathrm{d}}
\newcommand{\F}{\mathcal{F}}

\renewcommand{\P}{\mathbb{P}} 
\renewcommand{\H}{\mathcal{H}}
\renewcommand{\S}{\mathcal{S}}
\renewcommand{\L}{\mathcal{L}} 

\newcommand{\bbR}{\mathbb{R}}

\newcommand{\cA}{\mathcal{A}}

\newcommand{\bes}{\begin{equation*}}
\newcommand{\ees}{\end{equation*}}
\newcommand{\beas}{\begin{eqnarray*}}
\newcommand{\eeas}{\end{eqnarray*}}
\newcommand{\bea}{\begin{eqnarray}}
\newcommand{\eea}{\end{eqnarray}}
\newcommand{\be}{\begin{equation}}
\newcommand{\ee}{\end{equation}}
\newcommand{\bei}{\begin{itemize}}
\newcommand{\eei}{\end{itemize}}
\newcommand{\bec}{\begin{cases}}
\newcommand{\eec}{\end{cases}}
\newcommand{\ben}{\begin{enumerate}}
\newcommand{\een}{\end{enumerate}}

\newcommand{\bbP}{\mathbb{P}}
\newcommand{\bbE}{\mathbb{E}}

\newcommand{\bbQ}{\mathbb{Q}}

\newcommand{\bbl}{\begin{block}}
\newcommand{\ebl}{\end{block}}

\newcommand{\De}{\mathrm{d}}

\newcommand{\rme}{\mathrm{e}}

\newcommand{\Id}{\mathrm{Id}}

\newtheorem{theorem}{Theorem}[section]
\newtheorem{proposition}[theorem]{Proposition}
\newtheorem{lemma}[theorem]{Lemma}
\newtheorem{corollary}[theorem]{Corollary}
\newtheorem{assumption}{Assumption}

\theoremstyle{definition}
\newtheorem{definition}[theorem]{Definition}
\newtheorem{example}[theorem]{Example}

\theoremstyle{remark}
\newtheorem{rem}{Remark}

\newcounter{paragraphe} 
\newcommand{\paranum}{\stepcounter{paragraphe}\textbf{\theparagraphe.}\space} 

\usepackage{authblk}

\title{\bf Propagation of weak log-concavity along generalised heat flows via Hamilton-Jacobi equations}

\author[1]{Louis-Pierre \textsc{Chaintron}}
\author[2]{Giovanni \textsc{Conforti}}
\author[3]{Katharina \textsc{Eichinger}}
\affil[1]{\small DMA, École normale supérieure, Université PSL, CNRS, 75005 Paris, France}
\affil[2]{\small Università degli Studi di Padova, Via Trieste, 63, 35131 Padova, Italia}
\affil[3]{\small Laboratoire de math\'ematiques d'Orsay, Universit\'e Paris-Saclay, Orsay, France \newline and ParMA, Inria Saclay, Palaiseau, France}

\date{\today}

\begin{document}

\maketitle

\abstract{A well-known consequence of the Prékopa-Leindler inequality is the preservation of log-concavity by the heat semigroup.
Unfortunately, this property does not hold for more general semigroups.
In this paper, we exhibit a slightly weaker notion of log-concavity that can be propagated along generalised heat semigroups.
As a consequence, we obtain log-semiconcavity properties for the ground state of Schrödinger operators for non-convex potentials, as well as propagation of functional inequalities along generalised heat flows.
We then investigate the preservation of weak log-concavity by conditioning and marginalisation, following the seminal works of Brascamp and Lieb.
To our knowledge, our results are the first of this type in non log-concave settings.
We eventually study  generation of log-concavity by parabolic regularisation and prove novel two-sided log-Hessian estimates for the fundamental solution of parabolic equations with unbounded coefficients, which can be made uniform in time.
These properties are obtained as a consequence of new propagation of weak convexity results for quadratic Hamilton-Jacobi-Bellman (HJB) equations.
The proofs rely on a stochastic control interpretation combined with a second order analysis of reflection coupling along HJB characteristics.}

\tableofcontents

\section{Introduction} \label{sec:Intro}

Let $(\S_t)_{t \geq 0}$ denote the semigroup generated by the generalised heat operator $\L$ whose action against smooth functions $\varphi : \R^d \rightarrow \R$ is given by 
\begin{equation} \label{eq:GeneralHeat}
\L \varphi := - \nabla U \cdot \nabla \varphi  + \frac{1}{2} \Delta \varphi - V \varphi, 
\end{equation} 
where $U, V : \R^d \rightarrow \R$ are smooth potentials whose derivatives have polynomial growth, with $\nabla U$ Lipschitz. 

\medskip

\paranum\textbf{Preservation of weak log-concavity.}
In their seminal work \cite[Section 6]{brascamp1976extensions}, Brascamp and Lieb proved that $( \S_t )_{t \geq 0}$ preserves log-concavity when $U \equiv 0$ and $V$ is uniformly convex: if $g$ is a positive smooth function with $\log g$ concave, then so is $\S_t g$.
Leveraging the structure of the heat kernel, this result follows from a consequence of the Prékopa-Leindler inequality: log-concavity is preserved under convolution \cite[Section 3]{prekopa1973logarithmic}.
This result was strengthened in \cite[Theorem 4.2]{brascamp1976extensions}, see also \cite{barthe1997inegalites,bennett2008brascamp} and references therein for a survey about the Brascamp-Lieb inequalities and their consequences.
A natural question is then the extension of the preservation of log-concavity for more general operators of type \eqref{eq:GeneralHeat}.
When $V \equiv 0$, \cite{kolesnikov2001diffusion} provided a negative answer in general, by showing that this property is equivalent to the fact that the potential $U$ is quadratic, identifying $\L$ as an Ornstein-Uhlenbeck generator.

The notion of log-concavity has thus to be relaxed for a preservation property to hold.
Following \cite{lindvall1986coupling,eberle2016reflection},
the \emph{weak semiconvexity profile} of a $\C^1$ function $\varphi : \R^d \rightarrow \R$ is defined as 
\begin{equation} \label{eq:Profile}
\kappa_{\varphi} (r) := r^{-2} \inf_{\vert x - \hat{x} \vert = r} \langle \nabla \varphi (x) - \nabla \varphi ( \hat{x}), x - \hat{x} \rangle, \qquad r >0. \end{equation}
For $\alpha \in \R$, $\kappa_\varphi \geq \alpha$ is equivalent to $\varphi$ being $\alpha$-semiconvex in the usual sense, i.e. $x \mapsto \varphi (x) - \tfrac{\alpha}{2} \vert x \vert^2$ being convex or equivalently $\nabla^2 \varphi \geq \alpha \mathrm{Id}$ in a weak sense.
However, working with $\kappa_\varphi$ allows for further flexibility in quantifying the convexity default of $\varphi$. Since
\[ \langle \nabla \varphi (x) - \nabla \varphi ( \hat{x}), x - \hat{x} \rangle = \big\langle x - \hat{x}, \int_0^1 \nabla^2 \varphi ( (1-t) \hat{x} + t x ) \d t \, [ x- \hat{ x} ] \big\rangle, \]
the weak semiconvexity profile $\kappa_\varphi$ can be seen as an ``integrated'' convexity modulus.
This notion is {isotropic} in the sense that $\kappa_\varphi$ only deals with with $\vert x - \hat{x} \vert$, see Section \ref{sec:mainHJB} for an anisotropic adaptation.
While convexity requires $\kappa_g \geq 0$ everywhere, we relax this notion as follows.

\begin{definition}[Asymptotic convexity]
A $\C^1$ function $\varphi : \R^d \rightarrow \R$ is \emph{asymptotically convex} if
 \[ \liminf_{r \rightarrow +\infty} \kappa_\varphi (r) \geq 0.\]   
$\varphi$ is \emph{strictly asymptotically convex} if this inequality is strict. 
\end{definition}

This class of functions is stable under Lipschitz perturbations, and encompasses (but is not restricted to) Lipschitz perturbations of convex functions.
Moreover, we emphasise that $\liminf_{r \downarrow 0} \kappa_\varphi (r) > -\infty$ implies that $\varphi$ is semiconvex in the usual sense.

Asymptotic convexity has been exploited in \cite{lindvall1986coupling,chen1989coupling,eberle2016reflection} and many following works for its application to exponential contraction for diffusion operators, through the probabilistic notion of \emph{reflection coupling}.   
More recently, a second-order analysis of this construction has been leveraged in \cite{conforti2024weak} to establish weak semiconvexity of Schrödinger's potentials in the framework of entropic optimal transport.

Let us now state our first result on preservation of weak log-concavity, which seems to be the first one that can handle non log-concave settings. 
Let the \emph{reciprocal potential} be defined by
\begin{equation} \label{eq:ReciPotential}
{\A_U} := \frac{1}{2} \vert \nabla U \vert^2 - \frac{1}{2} \Delta U, 
\end{equation}
this terminology being justified later in this introduction.
The statements of our results always implicitly assume that $\kappa_{\A_U + V}$ is everywhere finite.
This is a mild assumption, which holds e.g. if $\nabla[ \A_U + V]$
has a one-sided uniform continuity modulus.

\begin{theorem} [Weak semiconvexity propagation] \label{thm:HeatLogConcave}
If $f : \R_{>0} \rightarrow \R$ is any $\C^2$ function satisfying 
\begin{equation} \label{eq:IneqfIntro}
\forall r >0, \quad f''(r) \geq \frac{1}{2} f f' (r) - \frac{r}{2} \kappa_{{\A_U} + V} (r), \qquad \limsup_{r \downarrow 0} f (r) \leq 0,
\end{equation} 
then for any $\varphi$ in $\C^2 ( \R^d, \R_{>0} )$, with $-\log \varphi$ having polynomial growth derivatives,
\begin{equation} \label{eq:IntroProp}
\forall r >0, \; \kappa_{U-\log \varphi} (r) \geq r^{-1} f (r) \quad \Rightarrow \quad \forall t \geq 0, \forall r > 0, \; \kappa_{U-\log \S_t \varphi} \geq r^{-1} f (r). 
\end{equation}  
If $r \mapsto r \kappa_{\A_U + V} (r)$ is uniformly lower bounded and $\liminf_{r \rightarrow + \infty} \kappa_{\A_U + V} (r) \geq \alpha > 0$, then for every $\varepsilon >0$, there exists a profile $f$ satisfying \eqref{eq:IntroProp} and $\liminf_{r \rightarrow + \infty} r^{-1} f (r) \geq \sqrt{\alpha - \varepsilon}$.
\end{theorem}

The message of Theorem \ref{thm:HeatLogConcave} is that $( \S_t )_{ t \geq 0}$ preserves asymptotic strict log-concavity. 
The lower bound $\sqrt{\alpha}$ is sharp as it can be verified in the explicit Gaussian case $U \equiv 0$ and $V (x) = \tfrac{\alpha}{2} \vert x \vert^2$.
If $\alpha = 0$, a family of invariant profiles is given by $f(r) = - 2 \sqrt{L} \tanh [ \sqrt{L} r / 2]$, $L \geq 0$: since $f'(0) = - L$ is finite, usual semiconvexity is preserved in this case. 
Many other explicit examples are computed in Section \ref{ssec:ExistingLit}.
Invariant profiles satisfying \eqref{eq:IneqfIntro} are also given by $f (r) = g'(r)$, where $g$ is the solution of the ergodic Hamilton-Jacobi-Bellman equation
\[ \frac{1}{2} \vert g'(r) \vert^2 - 2 g'' (r) - \int_{0}^r u \kappa_{\A_U + V} (u) \d u = - \lambda, \quad r >0, \qquad g'(0) = 0, \]
for $\lambda > 0$, see e.g. \cite{arisawa1998ergodic} for an introduction to ergodic control problems.

Since the formal $L^2$-adjoint $\L^\star$ has the same shape \eqref{eq:GeneralHeat} as $\L$, an analogous result can be stated for the semi-group generated by $\L^\star$ -- we further notice that $\A_U + V$ is the same for $\L$ and $\L^\star$.
In particular, these results extend to the \emph{fundamental solution} of the parabolic equation $\partial_t \varphi + \L \varphi = 0$, see Theorem \ref{thm:Fondsol} below.

\medskip

\paranum\textbf{Ground state.} Preservation of log-concavity is a fundamental property with deep consequences in several mathematical fields. 
In spectral theory, one of them is the log-concavity of the ground state of $\L$, proved in \cite[Theorem 6.1] {brascamp1976extensions} when $U \equiv 0$.
In our general framework for weak semiconvexity, let us state an analogous result as an illustration of Theorem \ref{thm:HeatLogConcave}, which allows to handle non-convex situations. In particular, assuming $U \equiv 0$ for simplicity, we are able to show that $-\log\psi$ has a second derivative bounded below  even if $V$ does not. For instance, this happens when $V$ is the sum of a $a$-convex function and a $b-$Lipschitz function.

\begin{theorem}[Weak semiconvexity bounds for the ground state] \label{thm:groundState}
Let $U, V : \R^d \rightarrow \R^d$ be smooth potentials such that $\lim_{\vert x \vert \rightarrow + \infty} \A_U +V (x) = +\infty$.
Then $- \L$ has discrete spectrum with positive first eigenvalue, the related eigenspace being one-dimensional and generated by a positive never-vanishing ground state $\psi e^{-U}$.
\begin{itemize}
    \item If ${\A_U} + V$ is strictly asymptotically convex with bounded Hessian, then $U -\log \psi$ is strictly asymptotically convex with a lower bound on $\kappa_{U - \log \psi}$ that only depends on $\kappa_{\A_U +V}$. 
    \item If $\kappa_{\A_U + V} (r) \geq a -b / {r}$ for some $a ,b > 0$, then $\nabla^2 [ U - \log \psi ] \geq \min\{0,a^{1/2}-a^{-1}b^2,a^{1/2}-a^{-2}b\}$. 
\end{itemize}
\end{theorem}

In the first point, the bound on $\nabla^2 [ \A_U + V]$ is a mere technical assumption that we use to take limits in $\kappa_{-\!\log \S_t \varphi}$ as $t \rightarrow + \infty$.
As the proof shows, it implies a uniform bound on $\nabla^2 \log \psi$. 
The bound on $\nabla^2 [ \A_U + V]$ should not be necessary, since the lower bound on $\kappa_{U - \log \psi}$ does not depend on it.

\medskip

\paranum\textbf{Functional inequalities and conditional measures.} 
We now turn to the particular setting $V \equiv 0$.
Rather than $\L$ given by \eqref{eq:GeneralHeat}, we now study the $L^2$-adjoint $\L^\star$ seen as a Fokker-Planck type operator on probability measures.
Let us consider a flow $( \mu_t )_{t \geq 0}$ of probability measures satisfying
\begin{equation} \label{eq:DualFP}
\partial_t \mu_t = \L^\star \mu_t = \nabla \cdot [ \mu_t \nabla U + \tfrac{1}{2} \nabla \mu_t ], 
\end{equation} 
in weak sense. In this setting, $\mu_t$ corresponds to the law of a diffusion process with gradient drift $- \nabla U$.
We notice that $\L^\star$ has the shape $\eqref{eq:GeneralHeat}$ if we replace $(U,V)$ by $(-U,-\Delta U)$, so that Theorem \ref{thm:HeatLogConcave} can be applied. 
Let us state a key consequence of asymptotic log-concavity, extending results proved in \cite[Theorem 5.7]{conforti2023projected} and \cite{conforti2025coupling}.

\begin{theorem}[Gaussian image] \label{thm:GaussianImage}
Let $\mu$ be a strictly asymptotically log-concave probability measure. 
If $\kappa_{- \log \mu}$ has a continuous lower bound $\kappa : \R_{>0} \rightarrow \R$ such that $\liminf_{r \rightarrow + \infty} \kappa (r) > 0$ and $\int_0^1 \kappa^-(r)r \d r < +\infty$, then $\mu$ is the image of the standard Gaussian by a Lipschitz-continuous map.    
\end{theorem}

This property has strong consequences in terms of Gaussian concentration of measure and functional inequalities, see e.g. \cite{mikulincer2024brownian,fathi2024transportation}.
One of them is that any strictly asymptotically sufficiently regular log-concave measure satisfies the Log-Sobolev Inequality (LSI).
Let us now emphasize a straightforward property of weak log-concavity.

\begin{lemma}[Conditioning and semiconvexity] \label{eq:WeaklogCond}
If $\mu \in \ps ( \R^d \times \R^d)$ has a positive $\C^1$ density (still denoted by $\mu$), let
\[ \mu^{X \vert Y = y} ( x ) := \frac{\mu ( x,  y) \d x}{\int_{\R^d} \mu ( x', y ) \d x'}. \]
Then $\kappa_{-\!\log \mu^{X \vert Y = y}} \geq \kappa_{-\!\log\mu}$ for a.e. $y \in \R^d$.    
\end{lemma}

Combining these results gives the following important properties.

\begin{corollary} \label{cor:LSI}
If both $- \log \mu_0$ and $\A_U$ are strictly asymptotically convex and satisfy the assumptions of Theorem \ref{thm:HeatLogConcave}, \ref{thm:GaussianImage}, then:
\begin{itemize}
    \item For every $t \geq 0$ and $y \in \R^d$, $\mu_t$ and $\mu_t^{X \vert Y = y}$ are strictly asymptotically log-concave.
    \item For every $t \geq 0$ and $y \in \R^d$, $\mu_t$ and $\mu_t^{X \vert Y = y}$ satisfy LSI.
\end{itemize}
\end{corollary}

Up to our knowledge, Corollary \ref{cor:LSI} is the first simple criterion for proving LSI for conditional measures in non log-concave settings.
Beyond the entropic transport setting of \cite[Theorem 1.2]{conforti2023projected},
functional inequalities for conditional measures enjoy a renewed interest due to their usefulness for proving sharp propagation of chaos for interacting particle systems \cite[Theorem 2.8]{lacker2022quantitative}, \cite[Theorem 1]{ren2024size}.
However, applying Corollary \ref{cor:LSI} to this setting would require an additional effort to make the LSI constants dimension-free in the proof of Theorem \ref{thm:GaussianImage}.
By inspecting the proof of Theorem \ref{thm:GaussianImage}, the main challenge is to make the quantity $\int_0^1 \kappa^-(r) r \d r$ dimension-free.

\medskip

\paranum\textbf{Weak log-concavity and marginal measures.}
The preservation of log-concavity by the heat flow proved in \cite{brascamp1976extensions} was a mere consequence of the preservation of log-concavity by marginalisation already observed in \cite{prekopa1973logarithmic}.
In our setting of weak semiconvexity, the preservation by marginalisation is thus a natural question.
To wit, we adapt the setting of \cite[Theorem 4.3]{brascamp1976extensions} by considering a probability measure
\[ \pi ( \d x, \d y ) = \exp \big[ -h (x,y) - \tfrac{1}{2} [x-y]a^{-1}[x-y] \big] \d x \d y, \]
for a positive definite symmetric matrix $a \in \R^{d \times d}$.
Let us consider the marginal density
\[ \pi_0 (x) := \int_{\R^d} \exp \big[ -h (x,y) - \tfrac{1}{2} [x-y]a^{-1}[x-y] \big] \d y. \]
If $h$ is convex, \cite[Theorem 4.3]{brascamp1976extensions} proved that $x \mapsto e^{x \cdot D x} \pi_0 (x)$ is log-concave {for an explicit positive matrix $D$}.
To deal with the weak semiconvex setting, let us introduce 
\[ \varphi (x,z) := - \log \int_{\R^d} \exp \big[ -h (x,y) - \tfrac{1}{2} [z-y]a^{-1}[z-y] \big] \d y. \]
We observe that $(x,y) \mapsto h(x,y)$ being convex implies convexity for the partial function $x \mapsto h (x,y)$ for every $y \in \R^d$, and the analogous for $y \mapsto h(x,y)$. 
Similarly, our next result on marginalisation deals with the weak semiconvexity of partial functions.

\begin{theorem}[Partial semiconvexity] \label{thm:PartialConv}
For every $x,\hat{x} \in \R^d$, let $g^{x,\hat{x}} : \R_{> 0} \rightarrow \R$ be a $\C^{2}$ function such that for every $r > 0$,
\begin{equation*}
\overline{\sigma}_a^2 \partial_{rr} g^{x,\hat{x}} (r) \geq \frac{1}{2} g^{x,\hat{x}} (r) \partial_r g^{x,\hat{x}} (r), \quad \partial_r g^{x,\hat{x}} (r) \leq 0, \quad \limsup_{r \downarrow 0} g^{x,\hat{x}} (r) \leq 0,
\end{equation*} 
where $\overline\sigma_a > 0$ is the smallest eigenvalue of $a$.
We assume that for every $x \neq \hat{x}, z \neq \hat{z}$ in $\R^d$, 
\[ 
\begin{cases}
\langle a [ \nabla_z h ( x, z ) - \nabla_z h ( \hat{x}, \hat{z} ) ] ,  z - \hat{z} \rangle \geq g^{x,\hat{x}} ( \vert z - \hat{z} \vert ) \vert z - \hat{z} \vert, \\ \langle a [ \nabla_x h ( x,z ) - \nabla_x h ( \hat{x}, \hat{z} ) ] , x - \hat{x} \rangle \geq g^{x,\hat{x}} ( \vert z - \hat{z} \vert ) \vert x - \hat{x} \vert.
\end{cases} \]
Then for every $x \neq \hat{x}, z \neq \hat{z}$ in $\R^d$, we have 
\[ 
\begin{cases}
\langle a [ \nabla_z \varphi (x, z ) - \nabla_z \varphi ( \hat{x}, \hat{z} ) ], z - \hat{z} \rangle \geq g^{x,\hat{x}} ( \vert z - \hat{z} \vert ) \vert z - \hat{z} \vert, \\ \langle a [ \nabla_x \varphi ( x, z ) - \nabla_x \varphi ( \hat{x}, \hat{z} ) ] , x - \hat{x} \rangle \geq g^{x,\hat{x}} ( \vert z - \hat{z} \vert ) \vert x - \hat{x} \vert.
\end{cases} \]
In particular, when $a = \mathrm{Id}$, this implies $\kappa_{-\!\log \pi_0} (r) \geq 2 r^{-1} \inf_{\vert x - \hat{x} \vert = r} g^{x,\hat{x}}(r)$.
\end{theorem}

When $g \geq 0$, our assumptions are stronger than convexity and may seem less explicit than \cite[Theorem 4.3]{brascamp1976extensions}. However, the main strength of this result is that \emph{it does not require convexity} for $-\log \pi$ in general, allowing for settings like  $-\log \pi (x,y) = \tfrac{1}{2} \vert x - y \vert^2 + k ( x- y)$ for Lipschitz functions $k$ with bounded Hessian, see Theorem \ref{thm:ProPJoint} below.
To our knowledge and understanding, this result is the first broadly applicable sufficient condition that allows to study the log-semiconcavity of marginal laws of non log-concave joint densities.

\medskip

\paranum\textbf{Log-Hessian bounds on fundamental solutions.}
We have been focusing so far on \emph{propagation} of weak log-concavity for an initially data that is itself weakly log-concave.
Let us now turn to \emph{generation} of weak log-concavity: since parabolic equations have a regularising effect, we can indeed expect log-concavity to be generated by the equation, as it is well-known in the explicit Gaussian case of the heat kernel for $U = V \equiv 0$. 
These regularising effects are exemplified by the properties of the fundamental solution of the equation.
To give a complete picture, we added log-convexity estimates to our weak log-concavity results.

\begin{theorem}[Fundamental solution] \label{thm:Fondsol}
Let $G_t(x,y)$ be the fundamental solution satisfying, for every $x,y \in \R^d$,
\begin{equation*} 
\partial_t G_t(x,y) = \nabla \cdot [ G_t (x,y) \nabla U (x) + \tfrac{1}{2} \nabla G_t (x,y) ] - V (x) G_t(x,y), \quad t >0,
\end{equation*} 
with $G_t(x,y) \d x \xrightarrow[t \downarrow 0]{} \delta_{y} ( \d x)$.
Let $f$ be a $\C^2$  function with sub-quadratic growth such that $\int_0^1 f^-(r) \d r < + \infty$ and \eqref{eq:IneqfIntro} holds.
If furthermore $- \Lambda \mathrm{Id} \leq \nabla^2 [ \A_U + V] \leq \Lambda \mathrm{Id}$, then
\begin{itemize}
    \item \emph{Li-Yau-Hamilton inequality:} $\forall t >0, \quad -\nabla^2_x [ U + \log G_t(\cdot, y) ] \leq [ t^{-1} + \tfrac{t}{2} \Lambda ] \mathrm{Id}$. 
    \item \emph{Weak semiconvexity:} $\quad\quad\quad \forall t,r >0, \quad \kappa_{-U - \log G_t (\cdot,y)} (r) \geq r^{-1} f(r).$ 
\end{itemize}
In particular, $\nabla^2_x \log G_t$ is uniformly bounded if $\liminf_{r \downarrow 0} r^{-1} f(r) > - \infty$.
Furthermore, this bound is independent of $t \in [1,+\infty)$ if $\liminf_{r \rightarrow + \infty} r^{-1} f(r) > 0$. 
\end{theorem}

Many examples of explicit profiles $f$ are given in Section \ref{ssec:ExistingLit}, satisfying both $f '(0) \leq 0$ and $\liminf_{r \rightarrow + \infty} r^{-1} f(r) \geq 0$.
These two-sided estimates complete the picture of Gaussian estimates for fundamental solutions with unbounded coefficients, which has enjoyed many developments since the seminal work \cite{aronson1967parabolic}.
We also refer to \cite{fleming1985stochastic,sheu1991some} for a stochastic control interpretation of these estimates.
In the specific case of a semiconvex $\A_U + V$, Propositions \ref{pro:SemiGene+}-\ref{pro:SemiGene} further quantify the generation of (usual) semiconvexity starting from non-semiconvex initial data.

Since it does not depend on the initial condition, the proof of the Li-Yau inequality given in Section \ref{ssec:Proof intro} extends to the solution $\mu_t (x) = \int_{\R^d} G_t (x,y) \d \mu_0 (y)$ of \eqref{eq:DualFP} starting from any $\mu_0$.
If $-\log \mu_0$ is strictly asymptotically convex, we can combine this with Theorem \ref{thm:HeatLogConcave} to deduce two-sided log-Hessian estimates.
As a noteworthy example, we can handle initial data like $\mu_0 (x) \propto \exp [ - x^4 - g (x) ]$ for any semiconvex function $g$.

The above properties are isotropic: they only depend on the radial variable $r$.
We refer to Lemma \ref{lem:Fondsol} and Remark \ref{rem:Refine} for anisotropic refinements.
Furthermore,
the lower bound assumption on $\nabla^2 [ \A_U + V]$ can be relaxed for the matrix Li-Yau inequality as explained in Remark \ref{rem:Refine}. 
We prove this inequality following the seminal works \cite{li1986parabolic,hamilton1993matrix}. 
Integrating along the path between two points, the parabolic Harnack inequality classically follows.

\medskip

\textbf{The reciprocal potential.}
Before concluding this introduction, we would like to emphasize the crucial role of the reciprocal potential \eqref{eq:ReciPotential} throughout this article.
To simplify the discussion, let us restrict ourselves to the dual setting \eqref{eq:DualFP} with $V \equiv 0$. 
If we conjugate $\L^\star$ by the square-root of its invariant density, we obtain the operator
\begin{equation} \label{eq:Schr}
e^{ U } \L^\star [ e^{-U} \varphi ] = \frac{1}{2} \Delta \varphi - \A_{U} \varphi, 
\end{equation}
which is known as the Witten Laplacian with potential $\A_U$ in the literature of semi-classical analysis \cite{nier2005hypoelliptic}.
In view of the spectral theory for Schrödinger operators, it is thus natural that the assumptions of Theorem \ref{thm:groundState} deal with $\A_U$ rather than $U$.

A deeper insight is given by the second order calculus on Wasserstein space \cite{otto2001geometry}.
From the seminal work \cite{jordan1998variational}, it is well-known that the solution $( \nu_t )_{t \geq 0}$ of the heat equation $\partial_t \nu_t = \frac{1}{2} \Delta \nu_t$ can be seen as the Wasserstein gradient of (half of) the entropy functional within the heuristics of ``Otto's calculus'' \cite{otto2001geometry}, which sees the Wasserstein space $( \ps_2 (\R^d ), W_2)$ as an infinite-dimensional Riemannian manifold. 
If we differentiate this gradient flow once again in time, we can write the Newton equation on Wasserstein space \cite{conforti2019second,gentil2020dynamical,gigli2021second}
\[ \ddot{\nu}_t = \tfrac{1}{8} \nabla_{W_2} I ( \nu_t ), \]
where $I ( \nu_t ) := \int_{\R^d} \vert \nabla \log \nu_t \vert^2 \d \nu_t$ is the Fisher information.
The l.h.s. $\ddot{\nu}_t$ corresponds to the \emph{acceleration} in the Wasserstein manifold. 
This notion is made rigorous and connected to the Schrödinger bridge problem in \cite{conforti2019second}.

If we now perturb the heat equation  $\partial_t \nu_t = \frac{1}{2} \Delta \nu_t$ into \eqref{eq:DualFP}, then it follows from \cite[Theorem 1.6]{conforti2019second} that
\begin{equation} \label{eq:NewtReciPot}
\ddot{\mu}_t = \nabla_{W_2} [ \tfrac{1}{8}I + \hat{\A}_{U} ] ( \mu_t ), 
\end{equation} 
where $\hat{\A}_U : \mu \mapsto \int_{\R^d} {\A}_U \d \mu$ is a linear functional associated to $\A_U$, whose Wasserstein gradient is the vector field $\nabla \A_U$.
The Newton-like equation \eqref{eq:NewtReciPot} enforces the vision of $\A_U$ as component of the Wasserstein-acceleration associated with $\nabla U$ in \eqref{eq:DualFP}. 
It is thus natural that its behaviour is determining for \eqref{eq:DualFP}.
We notice that the second order equation \eqref{eq:NewtReciPot} is invariant under time-reversal, contrary to the heat equation which is well-posed in forward time only.
Recently, $\A_U$ was noticed to play a key role in the second order approximation of Schrödinger bridges by diffusion processes \cite{agarwal2024iterated,agarwal2025langevin}.

The interpretation of $\A_U$ as an acceleration has been further developed by the literature on stochastic mechanics, following attempts to give a formulation of quantum mechanics through diffusion processes \cite{schrodinger1932theorie,bernstein1932liaisons,nelson2020quantum}. 
Several notions of \emph{stochastic acceleration} have been developed \cite{zambrini1985stochastic,thieullen1993second,krener1997reciprocal}, identifying $\nabla \A_U$ as the acceleration of the stochastic diffusion with drift $-\nabla U$.
The article \cite{conforti2018couplings} computed that the pathwise density of such a diffusion with respect to the Wiener measure is 
\[ ( x_t )_{0 \leq t \leq T} \mapsto \exp \bigg[ - \int_0^T \A_U ( x_t ) \d t \bigg] , \]
and leveraged this fact to prove contraction estimates and functional inequalities for Brownian bridges, under convexity assumptions on $\A_U$.
In this perspective, Corollary \ref{cor:LSI} can be seen as a partial extension of these results to the case of a weakly semiconvex $\A_U$.

The terminology ``reciprocal potential'' is a slight variation on the ``reciprocal characteristic'' introduced in  \cite{conforti2018couplings} to denote $\nabla \A_U$. 
Its motivation comes from the literature on \emph{reciprocal measures}, see the survey article \cite{leonard2014reciprocal} and references therein.
Reciprocal measures are probability measures on path space that satisfy a suitable generalisation of the Markov property, which is invariant under time-reversal. 
In the diffusion setting, they basically amount to measures on path space whose marginal laws satisfy \eqref{eq:NewtReciPot}, with an additional degree of freedom in the endpoints.
Since our results mostly rely on the convexity profile of the scalar field $\A_U$, we believe the terminology ``reciprocal potential'' to be better adapted here.

\medskip

\textbf{Outline of the article.} 
The article is organised as follows.
Our key results on Hamilton-Jacobi-Bellman (HJB) equations are stated in Section \ref{ssec:HJB}.
They encompass new well-posedness results for solutions with polynomial growth, propagation of weak semi-convexity in time-dependent anisotropic settings, as well as novel uniform-in-time Hessian bounds in quadratic-like settings.
Many examples are detailed in Section \ref{ssec:ExistingLit}, providing robust and explicit construction of semiconvexity profiles.
A detailed comparison is further done with recent results from \cite{brigati2024HeatFlowLogconcavity,gozlan2025global}.
Section \ref{ssec:Proof intro} is devoted to the proof of the results stated in this introduction, using the log-transform to deduce them from the ones in Section \ref{ssec:HJB}.
Section \ref{ssec:ControlSto} sketches a detailed summary of the proof of the main results, explaining how we extend and strengthen the probabilistic construction from \cite{conforti2024weak} to propagate weak semiconvexity.
Some reminders and references on coupling by reflection are written there.
The results on HJB equations are proved in Section \ref{sec:ProofPropHJB}, whereas Appendix \ref{app:Explicit} details some computations from Section \ref{ssec:ExistingLit}.

To conclude this introduction, we would like to emphasise some perspectives and expectable outcomes from our results.
First, as mentioned above, semiconvexity estimates for HJB equations have recently been exploited \cite{fathi2020proof,chewi2023entropic,conforti2024weak} in the study of stochastic mass transport problems like the Schr\"odinger problem -- see \cite{leonard2013survey,nutz2021introduction} for an expository introduction to the subject -- to deduce regularity properties of dual optimizers, known as Schr\"odinger potentials. In this context, semiconvexity and semiconcavity bounds for Schr\"odinger potentials are particularly valuable, as they imply exponential convergence of Sinkhorn's algorithm \cite{chiarini2024semiconcavity} and related computational schemes \cite{shi2023diffusion}, as well as stability of optimal solutions \cite{divol2025tight}. One can reasonably hope that the finding of this article would spur further progress in these areas. 
In turn, these applications call for further exploration of the convexity properties of solutions of HJB equations in a more general setting, including Riemannian manifolds with lower bounded curvature, where some results already exist \cite{stroock1995estimate}, or kinetic equations.

\section{Main results for HJB equations and consequences} \label{sec:mainHJB}

In this section, we state our most general convexity results for Hamilton-Jacobi equations, we illustrate them and we use them to prove the ones stated in Section \ref{sec:Intro}.

\subsection{Time-dependent anisotropic weak semiconvexity} \label{ssec:HJB} 

We are able to cover the semigroup $(\S_t)_{t \geq 0}$ generated by the the following generalisation of \eqref{eq:GeneralHeat} whose action against smooth functions $\psi : \R^d \rightarrow \R$ is given by 
\begin{equation}\label{eq:GeneralHeat_a}
\L \psi := - a \nabla U \cdot \nabla \psi  + \frac{1}{2} \mathrm{Tr} [a\nabla^2\psi] - V \psi, 
\end{equation} 
where $U,V : [0,T] \times \R^d$, $h : \R^d \rightarrow \R$ are time-dependent potentials for a fixed time-horizon $T >0$, and $a \in \R^{d \times d}$ is a symmetric positive definite matrix.
Formally, the Hopf-Cole transformation $\varphi_t = -\log \S_{T-t}\psi$ turns it into the following HJB equation.
\begin{equation} \label{eq:HJB}
\begin{cases}
\partial_t \varphi_t - a \nabla U_t \cdot \nabla \varphi_t - \frac{1}{2} \nabla \varphi_t \cdot a \nabla \varphi_t + \frac{1}{2} \mathrm{Tr} [ a \nabla^2 \varphi_t ] +V_t= 0, \\
\varphi_T = h.
\end{cases}
\end{equation}
Our main result quantifies the weak semiconvexity of the classical solution $\varphi$ to the HJB equation.
To ensure that \eqref{eq:HJB} has a unique $\C^{1,2}$ solution, we make the following assumption throughout this section.

\begin{assumption}[Regularity] \label{ass:coeff} $\phantom{a}$
\begin{itemize}
    \item $U$ is $\C^{1,3}$, with $\nabla U$ globally Lipschitz.
    \item $\nabla U$ and $V$ (resp. $h$) are $\C^{1,2}$ (resp. $\C^2$) functions whose derivatives have polynomial growth.
    \item Either $\nabla V$ is globally bounded, or there exist $C >0$ and $\alpha > 1$ such that
    \begin{equation} \label{eq:Doubling}
    \forall (r,t,x,\hat{x}) \in [0,1] \times [0,T] \times \R^d \times \R^d, \quad \vert \nabla V_t ((1-r)x + r \hat{x}) \vert^\alpha \leq C [ 1 + V_t ( x ) + V_t ( \hat{x} ) ],   
    \end{equation} 
    and $h$ also satisfies one of these properties.
\end{itemize}
\end{assumption}

The last part \eqref{eq:Doubling} of Assumption \ref{ass:coeff} is convenient to handle situations where $V$ and $h$ have super-linear growth.
It is satisfied for instance by polynomials with odd degree.
Since \eqref{eq:Doubling} is not used in the proof of the weak convexity results, we can expect that it could be relaxed into a suitable Lyapunov-type condition guaranteeing well-posedness for \eqref{eq:HJB}.
We notice that $U_t + \varphi_t$ is still a solution of \eqref{eq:HJB} if we replace $(U,V,h)$ by $(0,V + \A,h + U)$, with $\A$ given by \eqref{eq:ReciPot} below.
Thus, it is always possible to make assumptions on $\A$ rather than $U$ if it is more convenient.

The following proposition is a variation on the results of \cite{krylov2008controlled}.

\begin{proposition}[Well-posedness] \label{pro:Well-Posed}
Under \ref{ass:coeff}, \eqref{eq:HJB} has a unique solution in $\C \cap W^{1,2}_{\mathrm{loc}} ([0,T] \times \R^d)$ with polynomial growth derivatives.
This solution is furthermore $\C^{1,2}$, with locally-Hölder derivatives.
\end{proposition}

To quantify the weak semiconvexity of $\varphi$ in this setting, we generalise \eqref{eq:Profile} into the \emph{anisotropic weak semiconvexity profile} 
\begin{equation} \label{eq:AnProfile}
\kappa^{A}_{\varphi_t} (r,e) := \inf_{\substack{x, \hat{x} \in \R^d \\ A ( x - \hat{x} ) = r e}} \vert A ( x - \hat{x} ) \vert^{-2} \langle Aa[\nabla \varphi_t (x ) - \nabla \varphi_t ( \hat{x} )] , A [ x - \hat{x} ] \rangle, \quad r >0, \, e \in \mathbb{S}^{d-1}, \end{equation} 
for any symmetric positive definite $A \in \R^{d \times d}$.
This profile fully keeps tracks of $x - \hat{x}$ through spherical coordinates, and allows for quantifying weak semiconvexity in weighted norm through $A$.
The adequate \emph{reciprocal potential} here reads
\begin{equation} \label{eq:ReciPot}
{\A}_U^a: = \frac{1}{2} \nabla U \cdot a \nabla U  - \frac{1}{2} \mathrm{Tr} [ a \nabla^2 U ] - \partial_t U, 
\end{equation} 
but we will drop the dependence on $U$ and $a$ and simply write ${\A}$ whenever it is clear from context.
In the following, we make one last technical assumption.

\begin{assumption} \label{ass:Coerciv}
Either $\nabla \varphi$ has linear growth, or $\min_{t \in [0,T]} \varphi_t (x) \rightarrow + \infty$ as $\vert x \vert \rightarrow + \infty$.
\end{assumption}

This assumption will allow for using $\varphi_t$ as a Lyapunov function.
From the representation formula \eqref{eq:Value} in Section \ref{ssec:ControlSto}, we can infer that a sufficient condition under \ref{ass:coeff} for $\min_{t \in [0,T]} \varphi_t (x) \rightarrow + \infty$ is $h(x) \geq c \vert x \vert - C$ for $c > 0$.
Alternatively, $\varphi$ is globally Lipschitz if $V$ and $h$ are Lipschitz, see e.g. \cite{chaintron2023existence}.

Under our running assumptions, the following key result quantifies propagation of weak semiconvexity along the equation.

\begin{theorem}[Weak semiconvexity for HJB] \label{thm:PropHJB}
Let $A \in \bbR^{d \times d}$ be symmetric positive definite, and denote by $\overline{\sigma}_A > 0$ the lowest eigenvalue of the positive definite square-root of $Aa A^\top$.
If $f : [0,T] \times \R_{>0} \times \mathbb{S}^{d-1} \rightarrow \R$ is a $\C^{1,2,1}$ function satisfying for every $(t,r,e)$,
\begin{equation} \label{eq:fIneq}
\begin{cases}
\partial_t f_t (r,e) + 2 \overline{\sigma}_A^2 \partial_{rr}^2 f_t (r,e) \geq f_t (r,e) \partial_r f_t (r,e) - r \kappa^{A}_{\cA_t + V_t}(r,e) + (4 r )^{-1} \vert \pi_{e}^\perp [ \nabla_e f_t ( r, e ) ] \vert^2, \\
\limsup_{(r',e') \rightarrow (0,e)} f_t (r',e' ) \leq 0,
\end{cases}
\end{equation} 
where $\pi^\perp_e$ is the orthogonal projection on the hyperplane with normal vector $e$, then
\[ \forall (r,e), \; \kappa^{A}_{U_T + h}(r,e)\geq r^{-1} f_T(r,e) \quad \Rightarrow \quad \forall t \in [0,T], \, \forall (r,e), \; 
\kappa^{A}_{U_t + \varphi_t}\geq r^{-1} f_t (r,e). \]
This result further holds if we replace $\kappa^{A}_{\A_t + V_t}$ in \eqref{eq:fIneq} by any lower bound on $\kappa^{A}_{\A_t + V_t}$.
\end{theorem}

The above result allows us to propagate anisotropic semiconvexity in a very general way.
Several examples are computed in Section \ref{ssec:ExistingLit}.
Let us start with a first purely angular example.

\begin{example}[Anisotropic convexity] \label{exm:Anisotropic}
If $ \kappa^{A}_{\A_t + V_t} (r,e ) \geq \langle e, \kappa_t e \rangle$ for a symmetric matrix $\kappa_t$, looking for specific solutions
$f_t (r,e) = r \langle e , \Sigma_t e \rangle$ reduces \eqref{eq:fIneq} to the matrix differential inequality
\[ \dot\Sigma_t \geq \Sigma_t \Sigma_t^\top - \kappa_t, \] 
which is reminiscent of Riccati equations.
This yields an anisotropic \emph{matrix} lower bound for $\nabla^2 \varphi_t$ that is indepedent of $\overline{\sigma}_A$, similar to the anisotropic extension of Caffarelli's contraction theorem proved in \cite[Theorem 5.4]{gozlan2025global} -- see Section \ref{ssec:ExistingLit} for further comparisons to \cite{gozlan2025global}. 
In the time-independent case $\kappa_t = \kappa$, we obtain that $\kappa^{1/2}$ is an invariant convexity profile.
This provides a matrix extension of the preservation of asymptotic convexity stated in Theorem \ref{thm:HeatLogConcave}. 
Allowing for more general matrix fields $\Sigma_t (r)$ would certainly provide improved convexity bounds.
\end{example}

When $U_t = U$ and $V_t = U$ do not depend on $t$, stationary solutions of \eqref{eq:fIneq} yield invariant convexity profiles.

\begin{corollary}[Invariant profile] \label{cor:Invar}
If $U$, $V$ are time-independent, then for any $\C^{2,1}$ function $f : \R_{>0} \times \mathbb{S}^{d-1} \rightarrow \R$ satisfying, for every $(r,e)$, 
\begin{equation} \label{eq:fIneqstatio}
\begin{cases}
2 \overline{\sigma}_A^2 \partial_{rr}^2 f (r,e) \geq f (r,e) \partial_r f (r,e) - r \kappa^{A}_{\cA + V}(r,e) \d t + (4 r )^{-1} \vert \pi_{e}^\perp [ \nabla_e f ( r, e ) ] \vert^2,  \\
\limsup_{(r',e') \rightarrow (0,e)} f (r',e' ) \leq 0,
\end{cases}
\end{equation} 
we have,
\[ \forall (r,e), \; \kappa^{A}_{U + g}(r,e)\geq r^{-1} f (r,e) \quad \Rightarrow \quad \forall t \in [0,T], \, \forall (r,e), \; 
\kappa^{A}_{U + \varphi_t}\geq r^{-1} f (r,e). \] 
We then say that $f$ is an invariant weak semiconvexity profile for \eqref{eq:HJB}.
\end{corollary}

In the isotropic setting, such a stationary profile is built in the proof of Theorem \ref{thm:HeatLogConcave} in Section \ref{ssec:Proof intro} -- see also \eqref{eq:Statio0th} in Section \ref{ssec:ExistingLit}. 

Let us state another useful consequence of our results, proving two-sided Hessian estimates for quadratic growth HJB equations.

\begin{corollary} \label{cor:Two-sided}
If $\nabla U$, $\nabla [ \A_U + V]$, $\nabla h$ are globally-Lipschitz, $h$ satisfies the requirements of Corollary \ref{cor:Invar} for $\sigma = A = \mathrm{Id}$ and $\int_0^1 \inf_e f^-(r,e) \d r < + \infty$, then $\nabla^2 \varphi_t$ is uniformly bounded.
If furthermore $\liminf_{r \rightarrow + \infty} \inf_e f(r,e) > 0$, then this bound is independent of $(t,T)$. 
\end{corollary}

Up to minor modifications, this result could be written in the general time-dependent setting of Theorem \ref{thm:PropHJB} as well. 
The bounds computed in the proof of Corollary \ref{cor:Two-sided} only depend on the Lipschitz constants of $\nabla U$, $\nabla V$, $\nabla h$;
it is thus reasonable to expect that \ref{ass:coeff}-\ref{ass:Coerciv} could be relaxed in this quadratic setting.

Given a $\C^2$ function $h : \R^d \times \R^d \rightarrow \R$ with polynomial growth derivatives, let us now introduce the Hamilton-Jacobi equation
\begin{equation} \label{eq:HJjoint}
\begin{cases}
\partial_t \varphi_t^x (z) - a \nabla_z U_t (z) \cdot \nabla_z \varphi_t^x (z) - \frac{1}{2} \nabla_z \varphi^x_t (z) \cdot a \nabla_z \varphi_t^x ( z ) + \frac{1}{2} \mathrm{Tr} [ a \nabla^2_z \varphi_t^x (z) ] + V_t (z) = 0, \\
\varphi_T^x (z) = h (x,z).
\end{cases}
\end{equation}
We then quantify the convexity profiles of $(x,z) \mapsto \varphi^x_t (z)$ with respect to each variable.

\begin{theorem}[Partial weak semiconvexity] \label{thm:ProPJoint}
In the setting of Theorem \ref{thm:PropHJB}, for every $(x,\hat{x}) \in \R^d \times \R^d$ let $f^{x,\hat{x}} : [0,T] \times \R_{>0} \times \mathbb{S}^{d-1} \rightarrow \R$ be a $\C^{1,2,1}$ function satisfying \eqref{eq:fIneq}.     
Let $g^{x,\hat{x}} : [0,T] \times \R_{> 0} \rightarrow \R$ be a $\C^{1,2}$ function such that for every $(t,r,e)$,
\begin{equation} \label{eq:gIneq}
\partial_t g^{x,\hat{x}}_t (r) + 2 \overline{\sigma}_A^2 \partial_{rr}^2 g^{x,\hat{x}}_t (r) \geq g^{x,\hat{x}}_t (r) \partial_r g^{x,\hat{x}}_t (r), 
\end{equation} 
as well as $f^{x,\hat{x}}_t (r,e) \geq g^{x,\hat{x}}_t (r)$, $\partial_r g^{x,\hat{x}} (r) \leq 0$ and $\limsup_{r \downarrow 0} g_t (r) \leq 0$. We assume that for every $x \neq \hat{x}, z \neq \hat{z}$ in $\R^d$, 
\[ 
\begin{cases}
\langle A a[ \nabla_z h ( x, z ) - \nabla_z h ( \hat{x}, \hat{z} ) ], A [ z - \hat{z} ] \rangle \geq f^{x,\hat{x}}_T ( r^z, e^z ) \vert A ( z - \hat{z} ) \vert, \\ \langle A a[ \nabla_x h ( x,z ) - \nabla_x h ( \hat{x}, \hat{z} ) ], A [ x - \hat{x} ] \rangle \geq g^{x,\hat{x}}_T ( r^z ) \vert A ( x - \hat{x} ) \vert,
\end{cases} \]
where $(r^z,e^z)$ is defined by $A ( z - \hat{z} )= r^z e^z$. Then for every $t,x \neq \hat{x}, z \neq \hat{z}$, we have 
\[ 
\begin{cases}
\langle A a [ \nabla_z \varphi_t^x ( z ) - \nabla_z \varphi_t^{\hat{x}} ( \hat{z} ) ], A [ z - \hat{z} ] \rangle \geq f^{x,\hat{x}}_t ( r^z, e^z ) \vert A ( z - \hat{z} ) \vert, \\ \langle A a [ \nabla_x \varphi_t^x ( z ) - \nabla_x \varphi_t^{\hat{x}} ( \hat{z} ) ], A [ x - \hat{x} ] \rangle \geq g^{x,\hat{x}}_t ( r^z ) \vert A ( x - \hat{x} )\vert.
\end{cases} \]
\end{theorem}

If $\kappa^A_{\A + v} \geq 0$, we notice that $f = g$ is a suitable choice provided that $\partial_r f \leq 0$, for instance $f (r) = g (r) = r - h_L (r)$ given by \eqref{eq:Statio0th} below.
This choice allows for handling non-convex functions like $h(x,y) = \tfrac{1}{2} \vert x - y \vert^2 + k (x-y)$ for Lipschitz functions $k$ with bounded Hessian.

To complete the picture, let us recall that \eqref{eq:HJB} enjoys an unconditional semiconcavity preservation property.

\begin{proposition} \label{pro:Semiconcave}
If $h$ is $\Lambda_h$-semiconcave and $\A_U + V$ is $\Lambda_\A$-semiconcave for $\Lambda_h, \Lambda_\A \geq 0$, then the solution $\varphi_t$ of \eqref{eq:HJB} is $\Lambda_h + (T-t) \Lambda_\A$-semiconcave.
\end{proposition}

The proof of this standard result is sketched in Remark \ref{rem:SConcavProp}.

\subsection{Examples and explicit convexity profiles} \label{ssec:ExistingLit}

This section computes several examples of isotropic convexity profiles and sketches comparisons with existing results.
Thus, we look for solutions of
\begin{equation} \label{eq:InterIneq} 
\begin{cases}
\partial_t f_t (r) + 2 \overline{\sigma}_A^2 \partial_{rr}^2 f_t (r) \geq f_t (r) \partial_r f_t (r) - r \kappa^{A}_{\cA_t + V_t}(r), \quad r > 0, \, t \in [0,T], \\
\limsup_{r \downarrow 0} f_t (r) \leq 0, \quad t \in [0,T].
\end{cases}
\end{equation} 
First, in the particular case $\kappa^A_{\A_t + V_t} (r) = - \beta_1 / r$ for some $\beta_1 >0$, we notice that
\[ f_t (r) = - \frac{r}{t + \alpha} + \frac{\beta_1}{2} (t+\alpha) - \frac{\beta_2}{t + \alpha}, \]
is a trivial admissible solution for any $\alpha > 0$, provided that $\beta_2 > 0$ is large enough so that $f_t(0) \leq 0$.
However, this linear solution does not leverage the diffusive term to improve the semiconvexity profile over time. 
Let us now design finer solutions to exploit the diffusion.

In the following, we will simply write $\kappa_t (r)$ instead of $\kappa^{A}_{\cA_t + V_t}(r)$, recalling that $\kappa_t (r)$ can actually be any lower bound on $\kappa^{A}_{\cA_t + V_t}(r)$ from Theorem \ref{thm:PropHJB}.
More restrictively, we are looking for $f$ solution of 
\begin{equation} \label{eq:exp_conv_prof}
\partial_t f_t(r)+2 \overline\sigma^{2}_A \partial^2_{rr} f_t(r) = f_t(r) \partial_r f_t (r)- r \kappa_t (r), \qquad f_t(0) = 0, 
\end{equation}
where $f_T$ is free for now -- remembering that we will have to check that $r \kappa_{U_T + h} (r) \geq f_T (r)$, $h$ being the terminal data for \eqref{eq:HJB}.

Let us start with a first informal statement, which proposes a systematic method for building solutions.

\begin{lemma}[Informal] \label{lem:time-dependent-prof-reg}
The solution of \eqref{eq:exp_conv_prof} is given by 
\begin{equation} \label{eq:HopfCEx}
f_t(r) = -4\overline\sigma^2_A \partial_r \log F_t (r),
\end{equation}
where we define
\begin{equation}
F_t (r) := \bbE\bigg[\exp\bigg[-\frac{1}{4\overline{\sigma}_A^2}\int_t^T K_s(|X^{t,r}_s|) \De s -\frac{1}{4\overline{\sigma}_A^2} \int_0^{| X^{t,r}_T|} f_T (u)\De u \bigg] \bigg],
\end{equation}
with $X^{t,r}_s := r + 2\overline{\sigma}_A B_{s-t}$ for a 1D Brownian motion $(B_s)_{s\geq 0}$ and 
$K_s(r) := \int_0^r \kappa_s(u)u \d u$.
\end{lemma}

To ease the reading, we do not specify the regularity and growth assumptions on $\kappa_t$ and $f_T$ at this stage, making the above statement informal.
Any assumption guaranteeing that $F_t$ is well defined and $\C^{1,2}$ would make this rigorous.

\begin{proof}
This is a consequence of the well-known Cole-Hopf transform \eqref{eq:HopfCEx}, which imposes
\begin{equation}
\partial_t F_t + 2\overline{\sigma}^2_A \partial_{rr}^2 F_t - \frac{1}{4\overline{\sigma}^2_A} F_t K_{t} = 0, \quad F_T(r) = \exp\bigg[-\frac{1}{4\overline{\sigma}^2_A}\int_0^{r}f_T(u)\De u \bigg], \quad F'_t(0) = 0.
\end{equation}
The conclusion follows from the Feynman-Kac representation formula.
\end{proof}

\textbf{Linear - hyperbolic profiles.} We now make the simplifying assumption that $\kappa_t (r) = \beta$ is constant, which amounts to assuming that $\A_t + V_t$ is $\beta$-convex, $\beta \in \R$. 
Using the Girsanov transform, it is then possible to make the transition density of an Ornstein-Uhlenbeck process appear in the expression of $F_t$, see Lemma \ref{lem:transitionOU} in Appendix \ref{sec:AppOU}.
When $\beta \in \{0, 1\}$, using that Ornstein-Uhlenbeck processes are time-changed Brownian motions, this was leveraged in \cite[Lemma 3.1]{conforti2024weak} and \cite[Lemma 5.9]{conforti2023projected} to compute explicit profiles under simplifying assumptions on $f_T$ involving the functions
\begin{equation} \label{eq:Statio0th}
h_{L}(r) := 2 \overline\sigma_A \sqrt{L} \tanh \bigg[ \frac{ \sqrt{L}}{2 \overline\sigma_A}r \bigg], 
\end{equation} 
for any $L > 0$, the $-h_L$
being the only stationary solutions of \eqref{eq:exp_conv_prof} without singularities when $\kappa_t \equiv 0$ from \cite[Table 1]{benton1972TableSolutionsOnedimensional}.

We now propose a generalisation of these computations by noticing that
\[ f_t (r) = A(t) r - B( t ) h_L [ B(t) r ]   \]
is a solution of \eqref{eq:exp_conv_prof} when $\kappa_t \equiv \beta$ as soon as
\[ \dot{A}(t) = A^2 (t) - \beta \quad \text{and} \quad \dot{B}(t) = A (t) B(t). \]
Let us give a few examples for $f_T (r) = \alpha r - h_L (r)$, which follow from tedious but straightforward computations that we omit here.
\begin{itemize}
    \item If $\beta \geq 0$ and $\alpha \geq -\sqrt{\beta}$, a suitable choice is \\
    \[{\textstyle A (t) =\sqrt{\beta} \frac{\alpha \cosh(\sqrt{\beta}(T-t)) + \sqrt{\beta} \sinh(\sqrt{\beta}(T-t))}{\sqrt{\beta} \cosh(\sqrt{\beta}(T-t)) + \alpha \sinh(\sqrt{\beta}(T-t))}, \quad B(t) = \frac{2\sqrt{\beta}}{(\sqrt{\beta}-\alpha)e^{-\sqrt{\beta}(T-t)} + (\sqrt{\beta}+\alpha)e^{\sqrt{\beta}(T-t)}}. }\]
    As $T-t \rightarrow +\infty$, the asymptotic profile reduces to $\sqrt{\beta}$ as in Theorem \ref{thm:HeatLogConcave}.
    If $\alpha < - \sqrt{\beta}$, this solution will explode when $T-t$ increases: semiconvexity is no more preserved in this regime. 
    \item If $\alpha =0$ and $\beta < 0$, a suitable choice is
    \[ A(t) = - \sqrt{-\beta} \tan ( \sqrt{-\beta} (T-t) ), \qquad
B(t) =  [ \cos ( \sqrt{-\beta} (T-t)) ]^{-1}. \] 
As above, the coefficients explode for $T-t$ large enough: semiconvexity is only preserved at short times.
\end{itemize}
When $\beta = 0$ and $\alpha \geq 0$, a convenient example is
\[ f_t(r) = \frac{\alpha}{1 + \alpha(T-t)} r - \frac{1}{1 + \alpha(T-t)} h_{L} \bigg[ \frac{r}{1 + \alpha(T-t)} \bigg], \]
fixing a typo in the rescaling in space from \cite[Lemma 3.1]{conforti2024weak}. Importantly, we notice that this example covers the convexity profile of functions $h$ that are Lipschitz perturbations with bounded Hessian of $\alpha$-convex functions, since in that case $\kappa_{h} (r) \geq \alpha - r^{-1} h_L (r)$ for a large enough $L$. 

\medskip

\textbf{Semiconvexity generation.}
The above examples satisfy $\liminf_{r \downarrow 0} r^{-1} f_t (r) > -\infty$ for all $t\geq T$.
In particular, $\kappa_{U_t + \varphi_t} (r) \geq r^{-1} f_t (r)$ as in Theorem \ref{thm:PropHJB} implies that $U_t + \varphi_t$ is \emph{semiconvex in the usual sense}, and then
\begin{equation} \label{eq:EigenBound}
\lambda_{\min} [ \nabla^2 [ U_t + \varphi_t ] ] = \liminf_{r \downarrow 0}  \kappa_{U_t + \varphi_t} (r). 
\end{equation}
On another side, if $U_T + h$ is only $L/2$-Lipschitz, then we only know that $\kappa_{U_T + h} (r) \geq -L/r$. 
Let us now quantify how semiconvexity when \emph{generated} along \eqref{eq:HJB} if $\A_t + V$ is convex.

\begin{proposition}[Semiconvexity generation] \label{pro:SemiGene+}
If $\kappa^{A}_{\cA_t + V_t}(r) \geq 0$ -- corresponding to $\kappa_t \equiv 0$ in \eqref{eq:exp_conv_prof} --, and 
$\liminf_{r\rightarrow \infty}\kappa_{U_T + g}(r) \geq \alpha > -\frac{1}{T}$, then
\[ \lambda_{\mathrm{min}}[ \nabla^2 [ U_t+\varphi_t ]] \geq \frac{\alpha}{1+(T-t)\alpha} +\frac{1}{(T-t)(1+(T-t)\alpha)} - \frac{1}{4\overline{\sigma}^2_A (T-t)^2} \bbE[X_\alpha^2],
\]
where $X_\alpha$ is a random variable with density
$$\rho (\d x) \propto \1_{x \geq 0} \exp \bigg[ -\frac{1}{4\overline{\sigma}_A^2} \int_0^{x} f_T (u) - \alpha u \, \d u \bigg] \exp \bigg[ -\frac{x^2}{8\overline{\sigma}^2_A (T_\alpha - t_\alpha)} \bigg], $$
with $T_\alpha := \frac{T}{1+ \alpha(T-t)}$ and $t_\alpha := \frac{t}{1+ \alpha(T-t)}$.
\end{proposition}

This result is proved in Appendix \ref{sec:appSC}.
In the Lipschitz case mentioned above this semiconvexity generation becomes the following.
\begin{proposition}[Lipschitz perturbation] \label{pro:SemiGene}
If $\kappa^{A}_{\cA_t + V_t}(r) \geq 0$ -- corresponding to $\kappa_t \equiv 0$ in \eqref{eq:exp_conv_prof} --, and $\kappa_{U_T + h}(r) \geq \alpha -\frac{L}{r}$ for $\alpha, L \geq 0$, then
\[
\lambda_{\mathrm{min}}[\nabla^2 [ U_t +\varphi_t ] ] \geq \frac{\alpha}{1+(T-t)\alpha}-\frac{L}{\sqrt{2\pi\overline{\sigma}^2(T-t)(1+(T-t)\alpha)^3}} -\frac{L^2}{4\overline{\sigma}^2(1+(T-t)\alpha)^2}. 
\]
\end{proposition}

The work \cite{brigati2024HeatFlowLogconcavity} also computed a lower bound for $\lambda_{\mathrm{min}}[ \nabla^2 \varphi_t ]$ when $U \equiv V \equiv 0$ and $\sigma = 1$. 
After simple manipulations, this bound reads
\begin{equation}
\lambda_{\mathrm{min}}[\nabla^2 \varphi_t] \geq \frac{\alpha}{1+(T-t)\alpha} - \frac{L}{\sqrt{(T-t)(1+(T-t)\alpha)^3}} - \frac{L^2}{4(1+(T-t)\alpha)^2} .
\end{equation}
This bound is qualitatively the same as ours, with slightly worse multiplicative constants.

\medskip

\textbf{Invariant semiconvexity profiles.} 
Given $a,b>0$, we consider the stationary version of \eqref{eq:InterIneq} for $\kappa_t (r) = a r - b$, 
\begin{equation}\label{diffineq_conv_lip}
2f''(r)-f(r)f'(r)+a r - b\geq 0.
\end{equation}
A linear solution is given by 
\begin{equation}\label{eq:linear_profile}
f_0(r) = a^{1/2} r-a^{-1/2} b.
\end{equation}
For any $0<\tau\leq 2 \frac{b_{f_0}}{a_{f_0}}= 2 b/a$ to be chosen later on, we consider the parabola $r\mapsto p(r)$ such that $p(0)=0$ and $p$ is tangent to $f_0$ at $\tau$. 
This parabola $p(r)=\frac12\alpha_{\tau}r^2+\beta_\tau r $ is given by
\begin{equation}\label{eq:tangent}
\begin{cases} \alpha_\tau \frac{\tau^2}{2}+\beta_\tau\tau=a^{1/2} \tau -a^{-1/2} b\\
\alpha_\tau \tau+\beta_\tau=a^{1/2}
\end{cases}, \quad \text{whence} \quad \begin{cases}\alpha_\tau=2a^{-1/2} b\tau^{-2}\\ \beta_\tau=a^{1/2} -2a^{-1/2} b\tau^{-1} \end{cases}.
\end{equation}

\begin{lemma}
Let $\tau= \min\{2b/a,2a^{1/2}/b, 2a^{3/2}b \}$. Then $r\mapsto p_{\tau}(r)$ satisfies the inequality \eqref{diffineq_conv_lip} on $[0,\tau]$.
\end{lemma}

\begin{proof}
By definition of $\tau$, we have $\alpha_\tau>0, \beta_\tau\leq0$, as well as $p_{\tau}(r)\leq0$ on $[0,-\beta_\tau/\alpha_\tau]$, with $-\beta_\tau/\alpha_\tau$ being the vertex of the parabola. Moreover, for $r\in[0,-\beta_{\tau}/\alpha_{\tau}]$,
\begin{equation*}
0\leq -p_{\tau}(r)\leq a^{-1/2} b, \quad 0\leq -p_{\tau}'(r) \leq -\beta_{\tau} \quad \text{on $[0,-\beta_\tau/\alpha_\tau]$},
\end{equation*}
using that the convex function $p(r)$ lies above its tangent at $r = \tau$ for the first inequality.
This gives
\begin{align}\label{eq:boring_calc}
2p''_{\tau}(r)-p_{\tau}p'_{\tau}(r)+a r-b &\geq 2 \alpha_\tau+   a^{-1/2} b \beta_{\tau} +ar-b\\
&\stackrel{\eqref{eq:tangent}}{\geq} 4a^{1/2} b \tau^{-2}+ a^{-1/2} b (a^{1/2}-2a^{-1/2}b\tau^{-1
}) + a r -b\\
&=2\tau^{-1}\underbrace{(2a^{1/2} b\tau^{-1}-a^{-1}b^2)}_{\geq 0\,\, \text{by choice of $\tau$}} + a r \geq0.
\end{align}
We have thus shown that the sought inequality \eqref{diffineq_conv_lip}  holds on $[0,-\beta_\tau/\alpha_\tau]$. Let us now show that it holds on   $[-\beta_{\tau}/\alpha_\tau,\tau]$. We distinguish two cases:
\begin{itemize}
    \item If $p(r)\leq0$ for $r \in [-\beta_{\tau}/\alpha_\tau,\tau]$, then we have $p(r)p'(r)\leq 0$ on the same interval. 
    The sought inequality then results from $2p''_{\tau}(r)-p_{\tau}p'_{\tau}(r)\geq 2\alpha_{\tau}\geq -\alpha r+b$, where the last bound follows from \eqref{eq:tangent}-\eqref{eq:boring_calc}.
\item Otherwise, we set $r^+ :=\inf\{r \geq - \beta_\tau/ \alpha_\tau, \; p_\tau(r)>0\}$.
From the first case, we deduce that \eqref{diffineq_conv_lip} holds in $[-\beta/\alpha_\tau,r^+]$, and we only need to show its validity on  $[r^+,\tau]$. 
On this interval, we have $p\geq 0,p'\geq 0,p''\geq0$, whence
\[
p_\tau p'_\tau(r) \leq p_\tau(r)p'_{\tau}(\tau) 
=(\alpha_{\tau}r^2/2+\beta_\tau r)f'_0(\tau)
=b\tau^{-2}r^2-2b\tau^{-1}r+a r\leq ar.
\]
Consequently, 
\begin{align*}
2p''_\tau(r)-p_\tau(r)p'_\tau(r)+a r-b &\geq 2\alpha_{\tau}-b =4 a^{-1/2} b\tau^{-2} -b \\
&\geq 4 a^{-1/2}b(1/2b^{-1} a)(1/2\,b a^{-1/2})-b=0,
\end{align*}
concluding the proof.
\end{itemize}
\end{proof}
To recover a stationary profile from this, we concatenate:  ,
\begin{equation} \label{eq:StatioGround}
f_{\tau}(r) = \begin{cases} p_{\tau}(r), \quad r\leq\tau,\\f_{0}(r),\quad r\geq \tau,\end{cases}
\qquad \tau = \min\{2b/a,2a^{1/2}/b, 2a^{3/2}/b \}.
\end{equation}
This profile $f_{\tau}$ is globally of class $\C^1$, convex, and $\C^2$ everywhere except at $r=\tau$. Moreover, it satisfies the desired inequality \eqref{diffineq_conv_lip} for $r<\tau$ because $p$ does and for $r>\tau$ because $f_0$ does. \\
Up to increasing $b$ a bit, we can assume that Inequality \eqref{diffineq_conv_lip} is strictly satisfied.
By locally smoothing $f_\tau$ on $[\tau/2,3\tau/2]$, we can thus produce a $\C^2$ profile $r \mapsto f(r)$ that satisfies \eqref{diffineq_conv_lip}, while coinciding with $f_\tau$ out of $[\tau/2,3\tau/2]$.
In particular, this profile encodes (usual) semiconvexity, because
\begin{equation} \label{eq:StatioSCG}
\liminf_{r \downarrow 0} r^{-1} f(r) \geq \min \{0,a^{1/2}-b^2/a, a^{1/2}-b^2 /a^2\}.
\end{equation} 

\medskip

\textbf{Non-constant convexity modulus.}
In \cite{gozlan2025global}, the convexity of a continuous function $f : \R^d \rightarrow \R$ is quantified using
\[ M^f_t (x_0, x_1) := (1-t) f (x_0) + t f (x_1) -f ( (1-t) x_0 + t x_1), \]
by defining
\[ \rho_f (r) := \inf \bigg\{ \frac{1}{t ( 1 -t )} M^f_t ( x_0, x_1 ) \; , 
\; t \in (0,1), \vert x_0 - x_1 \vert = r, (x_0, x_1) \in \R^d \times \R^d  \bigg\}. \]
An anisotropic version is also considered.
If $f$ is $\C^1$,
following \cite[Section 2.3]{gozlan2025global}, we can perform an expansion of $M_t^f (x_0,x_1) \geq t (1-t) \rho_f (r)$ around $t = 0$ to get
\[ f ( x_1) - f ( x_0 ) - \nabla f (x_0 ) \cdot ( x_1 - x_0) \geq \rho_f (r), \]
and the same holds when exchanging $x_0$ and $x_1$. Summing then yields
\[ \langle \nabla f ( x_1) - \nabla f (x_0 ), x_1 - x_0 \rangle \geq 2 \rho_f (r). \]
The inequality $\kappa_f (r) \geq 2 r^{-2} \rho_f (r)$ follows, showing that our convexity measurement $\kappa_f$ contains information similar to a rescaled version of $\rho_f$.

\medskip

To go one step further in the comparison, the key result \cite[Proposition 3.1]{gozlan2025global} shows that $\rho_{- \L_\varepsilon f } \geq - \rho^\star_f$, where $\rho^\star_f (r) := \sup_{h >0} r h- \rho_f (h)$ is the Legendre transform of $\rho_f$ and
\[ \L_\varepsilon f (x) := \varepsilon \log \int_{\R^d} \exp \big[ \varepsilon^{-1} [ \langle x, y \rangle - f(y) ] \big] \d y \]
is the entropic Legendre transform of $f$. 
We notice that $\rho^\star_f (r) = r^2/2 -\psi_1 (r)$, where $\psi_t$ satisfies the deterministic HJB equation 
\[ \partial_t \psi_t + \frac{1}{2} \vert \nabla \psi_t \vert^2 = 0, \qquad \psi_0 (r) = \rho_f (r) - r^2/2. \]
On the contrary, rather than this deterministic HJB equation, our convexity profile in Theorem \ref{thm:PropHJB} is propagated along a viscous Burgers equation, suggesting that we rather estimate the derivative of the convexity modulus $\rho_{-\L_\varepsilon f}$. This intuition can be made rigorous for polynomials, as the following computation shows.

Let now $\varphi^\varepsilon_t$ denote the solution of \eqref{eq:HJB} for $U \equiv V \equiv 0$, $T = 1$ and $h (x) = \varepsilon^{-1} [ f(x) - \vert x \vert^2 / 2 ]$.
Leveraging the log-transform and the explicit expression of the heat kernel, we get
\[ \varphi^\varepsilon_1 (x) = \log [ (2 \pi \varepsilon)^{d/2}] - \log \int_{\R^d} \exp \big[ - (2 \varepsilon)^{-1} \vert x- y \vert^2 - h (y) \big] \d y, \]
and then
\[ \L_\varepsilon g (x) = \tfrac{1}{2} \vert x \vert^2 - \varepsilon \varphi^\varepsilon_1 (x) + \varepsilon \log [ (2 \pi \varepsilon)^{d/2}]. \]
Since $\kappa^{ \mathrm{Id}}_{\varphi^\varepsilon_0} (r) = \kappa_{\varepsilon \varphi^\varepsilon_0} (r) \geq 2 r^{-2} \rho_f (r) -1$, we can apply Theorem \ref{thm:PropHJB} with $(a,A) = (\varepsilon \mathrm{Id},\mathrm{Id})$ and the explicit construction from Lemma \ref{lem:time-dependent-prof-reg}, which here reduces to
\[ h^\varepsilon_t (r) = - 4 \varepsilon \partial_r \log \E \bigg[ \exp \bigg[ - \frac{1}{4 \varepsilon} \int_0^{\vert r + 2 \sqrt\varepsilon B_{1-t} \vert} [2 u^{-1} \rho_f (u) - u] \d u \bigg] \bigg]. \]
We thus obtain that $r^2 \kappa_{- \L_\varepsilon g} (r) \geq r h^\varepsilon_0 (r) - r^2$.
To ease the comparison, a formal large deviation computation suggests that
\begin{align*} \lim_{\varepsilon \downarrow 0} r h^\varepsilon_0 (r) - r^2 &= - r^2 + r \partial_r \inf_{h > 0} \int_0^h [2 u^{-1} \rho_f (u) - u] \d u + \frac{1}{2} \vert h - r \vert^2 \\
&= r \partial_r \inf_{h > 0} \int_0^h [2 u^{-1} \rho_f (u) - u] \d u - rh. 
\end{align*}
The comparison with $ 2 \inf_{h >0} \rho_f (h) - rh$ is not obvious in general.
However, for a modulus of the kind $\rho_f (r) = r^p$, $p > 1$, we can verify that both these quantities are proportional to $r^{p/(p-1)}$.

\subsection{Proof of the results in Section \ref{sec:Intro}} \label{ssec:Proof intro}

\begin{proof}[Proof of Theorem \ref{thm:HeatLogConcave}]
Given $\varphi$ in $\C^2 (\R^d,\R_{>0})$ and $T >0$, let $\varphi_t := - \log \S_{T-t} \varphi$. 
We notice that $\varphi_t$ is the solution of \eqref{eq:HJB} for $a = \mathrm{Id}$ and the terminal data $h = - \log \varphi$.
The first part of Theorem \ref{thm:HeatLogConcave} then follows from Corollary \ref{cor:Invar} for $a = A = \mathrm{Id}$ and the isotropic case $f(r,e) = f(r)$.
Assuming that $\liminf_{r \rightarrow + \infty} \kappa_{{\A_U} + V} (r) \geq \alpha$, let us now construct suitable invariant profiles to show that $\liminf_{r \rightarrow + \infty} \kappa_{- \log \S_t \varphi} (r) \geq \sqrt{\alpha}$.
To do so, we leverage that $\kappa_{\A_U + V}$ in \eqref{eq:fIneq} can be replaced by any function $r \mapsto \kappa(r)$ that lower bounds $\kappa_{\A_U + V}$, and we proceed in several steps.

\medskip

\emph{\textbf{Step 1.} Linear part.}
For any $\varepsilon > 0$, we can choose such a continuous function $\kappa_\varepsilon$ such that $\kappa_\varepsilon (r) = \alpha - \varepsilon > 0$ for $r \geq r_\varepsilon^-$, $r_\varepsilon^- > 1$ being determined by $\kappa_{\A_U + V}$.
By assumption, we know that $r \kappa_{\A_U + V} (r) \geq -K$ for some $K > 0$.
Let us consider $r^+_\varepsilon > r_\varepsilon^-$ to be chosen later on.
For $r \geq r^+_\varepsilon$, a linear solution of
\[ 2 f'' (r) \geq f (r) f'(r) - r \kappa_\varepsilon (r) \]
is given by $f_\varepsilon (r) = (r - r^+_\varepsilon ) \sqrt{\alpha - \varepsilon}$.

\medskip

\emph{\textbf{Step 2.} Controlled blow-up for $r \in [0,\hat{r}^-_\varepsilon)$.}
Integrating for $r_\varepsilon^- \leq r \leq r^+_\varepsilon$, let us now look for a solution of
\begin{equation} \label{eq:bakcward ODe}
2 [ f' ( r ) -  f' (r_\varepsilon^+) ] =  \frac{1}{2} [ \vert f(r) \vert^2 - \vert f ( r_\varepsilon^+ ) \vert^2 ] + \int_r^{r^+_\varepsilon} r '\kappa_\varepsilon (r') \d r', \quad r \in [r_\varepsilon^-,r^+_\varepsilon],
\end{equation} 
where $\kappa_\varepsilon (r') = \alpha - \varepsilon$.
Setting $g(r) = f(r^+_\varepsilon- r)$,
we thus want $g$ to satisfy
\[ g' ( r ) = - \sqrt{\alpha - \varepsilon} - \frac{1}{4} \vert g ( r ) \vert^2 - \frac{\alpha - \varepsilon}{4} [ 2 r r_\varepsilon^+ -r^2 ], \quad r \in [0,r^+_\varepsilon - r^-_\varepsilon], \]
with the initial condition $g(0) = 0$.
The solution of the above ODE is explosive in finite time, with $g(r) \rightarrow - \infty$.
Furthermore, we can choose $r^+_\varepsilon - r_\varepsilon^-$ larger than this explosion time, because the last term on the r.h.s. is non-positive independently of this choice. 
Hence, for an arbitrary large $M >0$, we can find $\hat{r}^-_\varepsilon \in (0,r_\varepsilon^+ - r _\varepsilon^- )$ before the explosion time such that
\begin{equation} \label{eq:ODENeg}
g(\hat{r}^-_\varepsilon) \leq - M, \quad g'(\hat{r}^-_\varepsilon) \leq - M. 
\end{equation} 
\emph{\textbf{Step 3.} Prolongation for $r \in [\hat{r}^-_\varepsilon, \hat{r}^+_\varepsilon)$.}
From the definition of $(r_\varepsilon^-,r^+_\varepsilon)$, we recall that $\kappa_\varepsilon (r^+_\varepsilon - \hat{r}^-_\varepsilon) = \alpha - \varepsilon$.
Leveraging once again our degree of freedom in the choice of $\kappa_\varepsilon$, we modify $\kappa_\varepsilon$ so that $r \mapsto r \kappa_\varepsilon (r)$ reaches its minimum value $- K =  ( r^+_\varepsilon - \hat{r}^+_\varepsilon) \kappa_\varepsilon( r^+_\varepsilon - \hat{r}^+_\varepsilon)$ for $\hat{r}^+_\varepsilon > \hat{r}^-_\varepsilon$ with $\hat{r}^+_\varepsilon$ arbitrarily close to $\hat{r}^-_\varepsilon$.
We then extend $g$ beyond $\hat{r}^-_\varepsilon$ into a function $g_\varepsilon$ solution of
\begin{equation} \label{eq:Blowup}
g'_\varepsilon ( r ) = - \sqrt{\alpha - \varepsilon} - \frac{1}{4} \vert g_\varepsilon ( r ) \vert^2 - \frac{1}{2} \int_{r^+_\varepsilon - r}^{r^+_\varepsilon} r' \kappa_\varepsilon ( r' ) \d r'. 
\end{equation}
This ODE is well-posed on a small neighbourhood of $\hat{r}^-_\varepsilon$.
Since $\int_{r^+_\varepsilon - r}^{r^+_\varepsilon} r' \kappa_\varepsilon ( r^+_\varepsilon - r' ) \d r'$ can be bounded independently of how small is $\hat{r}^+_\varepsilon- \hat{r}^-_\varepsilon$, we can assume that $\hat{r}^+_\varepsilon$ belongs to this neighbourhood.
Up to increasing $\hat{r}^+_\varepsilon$ a bit, we can assume that \eqref{eq:ODENeg} still holds with $\hat{r}^+_\varepsilon$ in place of $\hat{r}^-_\varepsilon$.

\medskip

\emph{\textbf{Step 4.} Quadratic extension.}
We eventually choose $(r^+_\varepsilon - r) \kappa_\varepsilon (r^+_\varepsilon - r) = - K$ for $r \in [\hat{r}^+_\varepsilon,r^+_\varepsilon]$, and we extend $g_\varepsilon$ to $[\hat{r}^+_\varepsilon,r^+_\varepsilon]$ into the unique quadratic function such that $g_\varepsilon$ is $\C^2$.
Differentiating \eqref{eq:Blowup} and using \eqref{eq:ODENeg}, we get $g''_\varepsilon ( \hat{r}^+_\varepsilon ) < 0$.
Since furthermore $g'_\varepsilon ( \hat{r}^+_\varepsilon ) < 0$ and $g_\varepsilon ( \hat{r}^+_\varepsilon ) < 0$, $r \in [\hat{r}^+_\varepsilon,r^+_\varepsilon] \mapsto g_\varepsilon (r)$ is negative decreasing.
As a consequence,
\[ \forall r \in [0,r^+_\varepsilon - r^-_\varepsilon], \quad 2 g_\varepsilon'' (r) \geq - g_\varepsilon (r) g_\varepsilon'(r) - ( r^+_\varepsilon - r ) \kappa_\varepsilon ( r^+_\varepsilon - r ). \]
Setting $f_\varepsilon (r) := g_\varepsilon ( r^+_\varepsilon - r)$ for $r \in [0,r^+_\varepsilon - r^-_\varepsilon]$ and using $f_\varepsilon (0) = g ( r^+_\varepsilon) < 0$, we deduce that $f_\varepsilon$ satisfies the requirements of Corollary \ref{cor:Invar}. 
As desired, we have $\lim_{r \rightarrow +\infty} r^{-1} f_\varepsilon (r) = \sqrt{\alpha - \varepsilon}$.
\end{proof}

\begin{rem}[Sub-quadratic profile] \label{rem:SubQuad}
We notice that the function $f$ constructed in the proof of Theorem \ref{thm:HeatLogConcave} can be made sub-quadratic, provided that $\kappa_{\A_U + V}$ has a continuous lower bound that is integrable on $(0,1]$. 
\end{rem}

\begin{proof}[Proof of Theorem \ref{thm:groundState}]
Conjugating $-\L$ as in \eqref{eq:Schr}, we reduce the problem to the Schrödinger operator $\H := -\frac{1}{2} \Delta + \A_U + V$.
Under our confinement assumption on $\A_U + V$, the Schrödinger operator $\H$ is self-adjoint and has compact resolvent, so that $\H$ has discrete spectrum with positive first eigenvalue $\lambda > 0$ and ground state $\psi \in H^1 (\R^d)$ \cite[Chapter XIII]{ReedSimonIV}.
Moreover, we have the Rayleigh variational representation
\begin{equation} \label{eq:Rayleigh}
\lambda = \inf_{\substack{\phi \in H^1 (\R^d) \\ \lVert \phi \Vert_{L^2 ( \R^d)} = 1}} \int_{\R^d} \lvert \nabla \phi (x) \vert^2 + [\A_U + V] (x) \phi^2 (x), \end{equation} 
the infimum being realized by $\psi$.
Then, $\psi (x) > 0$ and the non-degeneracy of the eigenspace classically follows.
Furthermore, setting $\T_t := e^{-t \H}$, we have the spectral projection
\begin{equation} \label{eq:CVground}
\forall \phi \in L^2 (\R^d), \qquad e^{\lambda t} \T_t \phi \xrightarrow[t \rightarrow +\infty]{L^2} \langle \phi, \psi \rangle_{L^2 
( \R^d)} \psi. 
\end{equation}   
Since $\A_U + V$ is locally Hölder, $\psi$ is $\C^2$ from the Schauder regularity theory.

Let us prove the first point in Theorem \ref{thm:groundState}.
We first approximate the Dirac mass by Gaussian kernels $( \phi^\varepsilon)_{\varepsilon >0}$.
For any $T>0$, we then set $\varphi^\varepsilon_t := - \log e^{\lambda (T-t)} \T_{T-t} \phi^\varepsilon$, which is a solution of \eqref{eq:HJB} for the modified potentials $(\tilde{U},\tilde{V}) = (0,\A_U + V - \lambda)$, so that $\kappa_{\A_{\tilde{U}} + \tilde{V}} = \kappa_{\A_{U} + V}$.
Since $\A_U +V$ has a bounded Hessian, Remark \ref{rem:SubQuad} applies and we can use the strictly asymptotically profile built in the proof of Theorem \ref{thm:HeatLogConcave}, by choosing $\varepsilon$ small enough so that $- \log \phi^\varepsilon$ satisfies the initialisation condition from Corollary \ref{cor:Invar}.
As previously, we deduce that $\kappa_{\varphi^\varepsilon_t} (r) \geq r^{-1} f_\varepsilon (r)$ with $\liminf_{r \rightarrow + \infty} f_\varepsilon (r) > 0$.

From \eqref{eq:CVground}, we deduce the almost everywhere pointwise convergence of $\varphi_{T-t}$ towards $- \log [ \langle \phi, \psi \rangle_{L^2 
( \R^d)} \psi ]$ as $T \rightarrow + \infty$ along a sub-sequence. 
To deduce that $\psi$ is strictly asymptotically log-concave, since $\langle \phi^\varepsilon, \psi \rangle > 0$ for $\varepsilon$ small enough, it is sufficient to prove the pointwise convergence of $(x,\hat{x}) \mapsto \nabla \varphi^\varepsilon_{T-t} (x) - \nabla \varphi^\varepsilon_{T-t} ( \hat{x})$ along a further sub-sequence. 
From Corollary \ref{cor:Two-sided}, we deduce that $(x,\hat{x}) \mapsto \nabla \varphi_t ( x ) - \nabla \varphi_t ( \hat{x} )$ is Lipschitz and bounded on every compact set, uniformly in  $(t,T)$. 
The Arzelà-Acoli theorem then provides compactness for the uniform convergence on every compact set, concluding the proof.

We follow the same path to prove the second point in Theorem \ref{thm:groundState}.
However, since $\mathcal{V} := \A_U + V$ no more has bounded Hessian, we need to approximate it by a $\mathcal{V}_\varepsilon$ with bounded Hessian whose convexity profile satisfies the same bound $\kappa_{\mathcal{V}_\varepsilon} (r) \geq a - b/r$.
To do so, we choose $\mathcal{V}_\varepsilon$ as the solution at time $t = \varepsilon$ of the HJB equation
\begin{equation} \label{eq:REgularHJB}
\partial_t \mathcal{V}_t (x) + \frac{1}{2} \vert \nabla \mathcal{V}_t (x) \vert^2 - a^2 x + b a^{1/2} = \frac{1}{2} \Delta \mathcal{V}_t (x), \qquad \mathcal{V}_0 = \A_U + V.
\end{equation}
We then apply Corollary \ref{cor:Invar} to $t \mapsto \mathcal{V}_{1-t}$ and the profile \eqref{eq:StatioGround} constructed in Section \ref{ssec:ExistingLit}, with $(a^2, b a^{1/2})$ in place of $(a,b)$.
Since this profile $\overline{f}_\tau (r)$ was lower bounded by $f_0(r)$ therein, this proves that $\kappa_{\mathcal{V}_\varepsilon} (r) \geq a - b/r$ as desired.
Furthermore, since $\overline{f}_\tau (r)$ is differentiable at $0$, $\mathcal{V}_\varepsilon$ is semiconvex.
To get that $\mathcal{V}_\varepsilon$ is semiconcave, we can either use the Li-Yau inequality following the proof of Theorem \ref{thm:Fondsol}, or compute the quadratic explicit solution of \eqref{eq:REgularHJB}.

Since $\mathcal{V}_\varepsilon$ now has bounded Hessian, we can reason as previously and apply Corollary \ref{cor:Invar} with the profile $\overline{f}_\tau (r)$ from \eqref{eq:StatioGround} with $(a,b)$.
This shows the desired lower bound for the profile of $-\log \psi^\varepsilon$, where $\psi^\varepsilon$ is the ground state associated to $\mathcal{V}_\varepsilon$ rather than $\A_U +V$.
To obtain the result for $\psi$, we send $\varepsilon \rightarrow 0$ using a stability result for the ground state following from the variational formulation \eqref{eq:Rayleigh}.  
\end{proof}

\begin{proof}[Proof of Theorem \ref{thm:GaussianImage}]
We prove Lipschitz continuity for the heat flow map introduced by \cite{kim2012GeneralizationCaffarellisContraction}, which is based on the time-reversal of the Ornstein-Uhlenbeck process. 

For convenience of the reader, we outline the formal construction of the heat flow map and refer to the literature for details and rigorous justification. The heat flow map starting from the standard Gaussian measure $\gamma$ towards a probability measure $\mu_0$ is obtained from the Fokker-Planck equation 
\begin{equation}
\partial_t \mu_t(x) = \nabla \cdot [ x \mu_t (x) + \nabla \mu_t(x) ], \quad \mu_0 = \mu,
\end{equation} 
corresponding to the quadratic potential $U(x) = \vert x \vert^2 / 2$. 
As $t \rightarrow + \infty$, $\mu_t$ weakly converges towards $\gamma$,
yielding an interpolation in-between $\mu$ and $\gamma$. 
To construct a transport map out of this, we notice that the Fokker-Planck equation rewrites as the continuity equation
\begin{equation} 
\partial_t \mu_t = \nabla \cdot \big[ \mu_t \nabla \log \frac{\De \mu_t}{\De \gamma} \big] , \quad \mu_0 = \mu,
\end{equation}
recalling that $\gamma (x) \propto e^{- \vert x \vert^2 /2}$ denotes the standard Gaussian density.
Let $(S_t)_{t\geq 0}$ denote the family of flow maps $\R^d \rightarrow \R^d$ satisfying 
\begin{equation}
\frac{\De}{\De t} S_t = -\nabla\log \frac{\De \mu_t}{\De \gamma} \circ S_t, \quad S_0 = \mathrm{Id},
\end{equation}
so that $(S_t)_{\#}\mu_0 = \mu_t$. 
Taking $t \rightarrow \infty$ formally yields $(S_\infty)_{\#}\mu_0 = \gamma$. 
Upon inverting $S_\infty$, we get that the heat flow map $M := (S_\infty)^{-1}$ satisfies $M_{\#}\gamma = \mu_0$. 

Following \cite[Lemma 11]{mikulincer2023LipschitzPropertiesTransportation} the heat flow map is Lipschitz-continuous with constant
\begin{equation} \label{eq:KMLip}
\exp\left[-\int_0^\infty \lambda_{\mathrm{min}}\bigg[\nabla^2 \log \frac{\d \mu_t}{\d \gamma}\bigg] \d t \right]. 
\end{equation} 
To make this rigorous, we only have to justify that this quantity is finite, leveraging the lower bounds on $\lambda_{\mathrm{\min}}$ proved in Appendix \ref{sec:appSC}. 
We thus have to bound the lowest eigenvalue of $\nabla^2 \log \frac{\De \mu_t}{\De \gamma}$.
To do so, let $e^{-h} := \frac{\De \mu}{\De \gamma}$ denote the density of the strictly asymptotically convex measure $\mu$ w.r.t. the standard Gaussian $\gamma$. Following the proof of \cite[Lemma 5.9]{conforti2023projected}, we deduce that
\begin{equation}
    \log \frac{\De \mu_t}{\De \gamma}(x) = \log Q_t [ e^{-h} ] (x) = \log P_{1 - e^{-2t}} [e^{-h}](e^{-t}x),
\end{equation}
where $(Q_t)_{t \geq 0}$ denotes the Ornstein-Uhlenbeck semigroup satisfying $\partial_t \int \varphi \, \d \mu_t = \int Q_t [ \varphi ] \, \d \mu_t$, and $(P_t)_{t \geq 0}$ is the heat semigroup. 
We can now use our semconvexity generation results, since $\log P_{1 - e^{-2t}}(e^{-h})(e^{-t}x) =\varphi_{e^{-2t}}(e^{-t}x)$, where $\varphi$ is the solution to the HJB equation \eqref{eq:HJB} for $T=1$, $a = 2\mathrm{Id}$, $U \equiv 0$, $V \equiv 0$.
More precisely, bounding \eqref{eq:KMLip} amounts to lower bounding
\begin{align}
    \int_0^\infty \lambda_{\mathrm{min}} \bigg[ \nabla^2 \log \frac{\d \mu_t}{\d \gamma} \bigg] \d t 
    = \int_0^\infty e^{-2t}\lambda_{\mathrm{min}}\left(\nabla^2 \varphi_{e^{-2t}}\right) \d t 
    = \frac{1}{2}\int_0^1 \lambda_{\mathrm{min}}\bigg[ \nabla^2 \varphi_{1-s} \bigg] \d s.
\end{align}
Since $\kappa_{h} = \kappa_{-\log \mu} -1$ and the lower bound $\kappa$ of $\kappa_{-\log \mu}$ satisfies $\liminf_{r \rightarrow + \infty} \kappa (r) >0$, we can find $\alpha >-1$ such that $\liminf_{r \rightarrow \infty} \kappa_{h}(r) \geq \alpha$. 
Lemma \ref{lem:prop_lower_eigenvalue} then yields the lower bound
\begin{align}
    \lambda_{\mathrm{min}} \nabla^2 [\varphi_{1-t} ] \geq \frac{\alpha}{1+t\alpha} +\frac{1}{t(1+t\alpha)} - \frac{1}{8 t^2} \bbE[X_\alpha^2],
\end{align}
where $X_\alpha$ is a random variable with density
$$\rho (\d x) \propto \1_{x \geq 0} \exp \bigg[ \frac{1}{8} \int_0^{x} \bar{\kappa}^-(u)u \, \d u \bigg] \exp \bigg[ -\frac{x^2(1+\alpha t)}{16 t} \bigg], $$
and $\bar{\kappa}^-(r) :=  -\min [ 0, \kappa (r) - \alpha ]$. 
The first term in the lower bound of $\lambda_{\mathrm{min}} \nabla^2 [\varphi_{1-t} ]$ being integrable for $t \in (0,1)$, we focus on the two other ones. Changing variables,
\begin{align*}
\frac{1}{t(1+t\alpha)} - \frac{1}{8 t^2} \bbE[X_\alpha^2] &= \frac{1}{t(1+t\alpha)} \bigg\{ 1 - Z_\alpha^{-1} \int_0^\infty x^2 \exp \bigg[ \frac{1}{8} \int_0^{2\sqrt{2t_\alpha}x} \bar{\kappa}^-(u)u \, \d u -\frac{x^2}{2} \bigg] \De x\bigg\} \\
&\geq \frac{Z^{-1}}{ t(1+\alpha t)} \int_0^\infty x^2  \exp \bigg[ -\frac{x^2}{2} \bigg]\bigg[1 - \exp \bigg[ \frac{1}{8} \int_0^{2\sqrt{2t_\alpha}x} \bar{\kappa}^-(u)u \, \d u \bigg] \bigg] \d x,
\end{align*} 
where we use
$t_\alpha := \frac{t}{1+\alpha t}$, and 
\[ Z_\alpha := \int_0^\infty  \exp \bigg[ \frac{1}{8} \int_0^{2\sqrt{2t_\alpha}x} \bar{\kappa}^-(u)u \, \d u -\frac{x^2}{2} \bigg] \De x \geq \int_0^\infty\exp \bigg[ -\frac{x^2}{2} \bigg] \d x =:Z. \]
Set $F (y) := \exp [ \frac{1}{8} \int_0^{y} \bar{\kappa}^-(u)u \, \d u ]$. 
A Taylor development then yields
\begin{align}
    F(0) - F(y) = -\int_0^1 \frac{F(sy)}{8} \bar{\kappa}^-(sy)sy^2\De s \geq -\frac{F(y)y^2}{8}\int_0^1  \bar{\kappa}^-(sy)s\De s.
\end{align}
As a consequence,
\begin{align}
\int_0^1& \bigg\{ \frac{1}{t(1+t\alpha)} - \frac{1}{8 t^2} \bbE[X_\alpha^2] \bigg\} \d t \geq \int_0^1\frac{Z^{-1}}{ t(1+\alpha t)} \int_0^\infty x^2  \exp \bigg[ -\frac{x^2}{2} \bigg] \big[ F(0) - F(2\sqrt{2t_\alpha}x) \big] \d x \d t \\
\geq & -Z^{-1}\int_0^\infty x^4  \exp \bigg[ -\frac{x^2}{2} \bigg]\int_0^1 \frac{F(2\sqrt{2t_\alpha}x)}{ (1+\alpha t)^2} \int_0^1  \bar{\kappa}^-(2\sqrt{2t_\alpha}x s)s\De s\De t\De x.
\end{align}
The finiteness of
\begin{align}
    \int_0^1 \int_0^1  \bar{\kappa}^-(2\sqrt{2t_\alpha}x s)s\De s\De t 
    = \int_0^1 \frac{1}{2\sqrt{2t_\alpha}x}\int_0^{2\sqrt{2t_\alpha}x}  \bar{\kappa}^-(u)u\De u\De t
\end{align}
is now granted by our assumption on $\kappa^-$. 

Our assumptions on $\kappa$ further imply that $\int_0^{2\sqrt{2t_\alpha}x} \bar{\kappa}^-(u)u \, \d u$ has at most quadratic growth in $x$, so that the exponential factor provides integrability with respect to $x$.
This proves that \eqref{eq:KMLip} is finite as desired, completing the proof.
\end{proof}

\begin{proof}[Proof of Theorem \ref{thm:PartialConv}]
For $t \in [0,1)$ and $x,z \in \R^d$, we set
\[ \varphi^x_t (z) := \frac{1}{2} \log [ ( 2 \pi)^d (1-t)^d \det a]  - \log \int_{\R^d} \exp \big[ -h (x,y) - \tfrac{1}{2(1-t)} [y-z] a^{-1}[y-z] \big] \d y. \]
We notice that $z \mapsto e^{- \varphi^x_{1-t} (z)}$ satisfies the forward heat equation with diffusion matrix $a$ and initial data $z \mapsto h(x,z)$.
Using the log-transform as previously, we deduce that $\varphi^x_t$ satisfies the HJB equation \eqref{eq:HJjoint} with $U \equiv V \equiv 0$ and $T = 1$.
It is sufficient to prove that $(x,z) \mapsto \varphi^x_0 (z)$ satisfies the result stated in Theorem \ref{thm:PartialConv} for $(x,z) \mapsto \varphi (x,z)$, because these functions only differ by a constant. 
This follows from the particular case $f = g$ and $A = \mathrm{Id}$ in Theorem \ref{thm:ProPJoint}. 
\end{proof}

Before proving Theorem \ref{thm:Fondsol}, let us state a slightly more general result concerning the weak semiconvexity part.

\begin{lemma}[Fundamental solution] \label{lem:Fondsol}
Let $G_t(x,y)$ be the fundamental solution satisfying, for every $x,y \in \R^d$,
\begin{equation} \label{eq:dualFond}
\partial_t G_t(x,y) = \nabla \cdot [ G_t (x,y) a \nabla U (x) + \tfrac{1}{2} a \nabla G_t (x,y) ] - V (x) G_t(x,y), \quad t >0,
\end{equation} 
with $G_t(x,y) \d x \xrightarrow[t \downarrow 0]{} \delta_{y} ( \d x)$.
Let $f$ be sub-quadratic and satisfy \eqref{eq:fIneqstatio} for $A = \mathrm{Id}$. Then,
\[ \forall t >0, \forall (r,e) \in \R_{>0} \times \mathbb{S}^{d-1}, \quad \kappa^{\mathrm{Id}}_{-U - \log G_t (\cdot,y)} (r,e) \geq r^{-1} f(r,e). \]
\end{lemma}

\begin{proof}
We refer to the classical works \cite{aronson1967parabolic,chabrowski1970construction,besala1975existence} for the existence of $G_t$ in our unbounded setting with Lipschitz drift $\nabla U$.

We can approximate $\delta_{y}$ in weak sense by a family of Gaussian kernels $( \gamma^\varepsilon_y )_{ \varepsilon > 0}$.
The classical solution of \eqref{eq:dualFond} with this modified initial data is
\[ G^\varepsilon_t (x) := \int_{\R^d} G_t (x,y') \gamma_y
^\varepsilon (y') \d y', \quad t > 0. \]
Equation \eqref{eq:dualFond} has the shape \eqref{eq:GeneralHeat_a} for the modified potentials
$( \tilde{U}, \tilde{V}):= (-U,V-\mathrm{Tr}[a \nabla^2 U])$
We then notice that $\A^a_{\tilde{U}} + \tilde{V} = \A^a_{U} + V$.
Since $- \log \gamma^\varepsilon_y$ is arbitrarily strongly convex for $\varepsilon$ small enough and $f$ is sub-quadratic, $\kappa_{-U - \log \gamma^\varepsilon_y}$ satisfies the initialization condition from Corollary \ref{cor:Invar}.
Performing the log-transform, the result propagates to $G^\varepsilon_t$. 
Sending $\varepsilon \rightarrow 0$, we get the result for $G_t$.
\end{proof}

\begin{proof}[Proof of Theorem \ref{thm:Fondsol}]
The weak semiconvexity part is a particular case of Lemma \ref{lem:Fondsol}.
Similarly, the bound on $\nabla^2_x [ U + \log G_t ]$ follows from Corollary \ref{cor:Two-sided}.
Performing the time-shift $t \mapsto G_{1+t}$, Corollary \ref{cor:Two-sided} further gives the uniform bound for $t \in [1,+\infty)$.

It remains to prove the Li-Yau inequality.
As previously, we approximate the 
Dirac mass at $y$ by smooth Gaussian kernels $( \gamma^\varepsilon_y )_{\varepsilon >0}$ and we introduce $\varphi^\varepsilon_t := - \log \S_t \gamma^\varepsilon_y$, which satisfies
\[ -\partial_t \varphi^\varepsilon_t - \nabla U \cdot \nabla \varphi^\varepsilon_t - \frac{1}{2} \vert \nabla \varphi^\varepsilon_t \vert^2 + \frac{1}{2} \Delta \varphi^\varepsilon_t + V = 0. \]
Since $U, V, \gamma^\varepsilon$ are smooth, $\S_t \gamma^\varepsilon_y$ is smooth and so is $\varphi^\varepsilon_t$.
Up to replacing $(\varphi^\varepsilon_t,U,V)$ by $(U+\varphi^\varepsilon,0,V + \A_U)$, we can assume that $U \equiv 0$.
For $1 \leq i,j \leq d$, differentiating twice yields 
\begin{equation} \label{eq:SecondDiff} 
\big[ \partial_t + \nabla \varphi^\varepsilon_t \cdot \nabla - \frac{1}{2} \Delta \big] \partial^2_{i,j} \varphi^\varepsilon_t = - \partial^2_{i,j} V + \nabla \partial_i \varphi^\varepsilon_t \cdot \nabla \partial_j \varphi^\varepsilon_t. 
\end{equation}
For $h \in \R^d$ with $\vert h \vert =1$, we then set
\[ \psi_t := 1 + t \sum_{1 \leq i, j \leq d} \partial^2_{i,j} \varphi^\varepsilon_t h_i h_j = 1 + t \langle h, \nabla^2 \varphi^\varepsilon_t h \rangle. \]
Summing, we get
\begin{align*}
\big[ \partial_t + \nabla \varphi^\varepsilon_t \cdot \nabla - \frac{1}{2} \Delta \big] \psi_t &= \langle h, \nabla^2 \varphi^\varepsilon_t h \rangle + t \vert \nabla^2 \varphi^\varepsilon_t h \vert^2  - t \langle h, \nabla^2 V h \rangle \\
&\geq t^{-1} [ \psi_t - 1 ] + t \vert t^{-1} [ \psi_t - 1 ] \vert^2 - t \langle h, \nabla^2 V h \rangle \\
&\geq - t^{-1} \psi_t - t \Lambda,
\end{align*}
where we used the Cauchy-Schwarz inequality to get the first inequality, and we got rid of a square at the second one.
Setting $\tilde\psi_t := \psi_t + \tfrac{1}{2} \Lambda t^2$, we obtain that
\[ \big[ \partial_t + \nabla \varphi^\varepsilon_t \cdot \nabla - \frac{1}{2} \Delta \big] \tilde\psi_t \geq - t^{-1} \tilde{\psi}_t. \]
Since $\nabla^2 \varphi^\varepsilon_t$ is bounded from Corollary \ref{cor:Two-sided}, we can find $\delta > 0$ small enough so that $\tilde{\psi}_\delta\geq 0$.
The parabolic maximum principle then shows that $\tilde{\psi}_t \geq 0$ for every $t \geq \delta$. 
Sending $\delta \rightarrow 0$ and then $\varepsilon \rightarrow 0$ concludes.  
\end{proof}

\begin{rem}[Refinements] \label{rem:Refine}
First, we notice that our proof gives the more precise estimate
\[ \forall (h,x) \in \R^d \times \R^d, \qquad \langle h, \nabla^2 \varphi^\varepsilon_t (x) h \rangle \leq t^{-1} + \sup_{y \in \R^d} \tfrac{t}{2} \langle h, \nabla^2 [ \A_U + V ] (y) h \rangle, \]
which has an anisotropic flavour, similarly to Lemma \ref{lem:Fondsol}. 

Second, the lower bound on $\nabla^2 [ \A_U + V ]$ is only used to apply Corollary \ref{cor:Two-sided}, because we need a (lower) bound on $\nabla^2 \varphi^\varepsilon_t$ to ensure $\tilde{\psi}_\delta \geq 0$ for $\delta$ small enough.
Since this lower bound is not involved in the final estimate,
this assumption could be removed if we could regularise $\A_U + V$ in a way that allows to apply Corollary \ref{cor:Two-sided}.
For instance, if $\kappa_{\A_U + V} (r) \geq - r^{-1} h_L (r)$ where $-h_L$ is the stationary profile \eqref{eq:Statio0th} for the HJB equation with $U \equiv V \equiv 0$, a suitable regularisation is $- \eta \log [ \gamma_\eta \ast \exp [ - \eta^{-1} ( \A_U +V ) ] ]$, $\gamma_\eta$ being the heat kernel at time $\eta$.  
\end{rem}

\begin{rem}[Semiconcavity propagation] \label{rem:SConcavProp}
To prove semiconcavity propagation from the initial data, rather than generation, it is sufficient to integrate \eqref{eq:SecondDiff} against the backward flow solution of
\[ \partial_t \mu_t + \nabla \cdot [\mu_t \nabla \varphi^\varepsilon_t + \tfrac{1}{2} \nabla \mu_t ] = 0, \qquad \mu_T = \delta_x. \]
This yields
\[ \langle h, \nabla^2 \varphi^\varepsilon_T ( x) h \rangle \geq \int_{\R^d} \langle h, \nabla^2 \varphi^\varepsilon_0 ( x) h \rangle \d \mu^\varepsilon_0 (x) + \int_0^T \int_{\R^d} \langle h, -\nabla^2 V ( x) h \rangle \d \mu^\varepsilon_t (x) \d t, \]
and the lower bound on $\nabla^2 \varphi^\varepsilon_T$ follows from the ones on $\nabla^2 \varphi^\varepsilon_0$ and $-\nabla^2 V$.
\end{rem}

\subsection{Stochastic control interpretation and proof strategy} \label{ssec:ControlSto}

Let $(\Omega,\F,(\F_t)_{0 \leq t \leq T},\P)$ be a filtered probability space.
For $t \in [0,T]$ and any progressively measurable square-integrable $u := (u_s)_{t \leq s \leq T}$, we consider the controlled diffusion process given by the strong solution of
\[ \d X^u_s = - \sigma \sigma^\top \nabla U_s ( X_t^u ) \d s + \sigma u_s \d s + \sigma \d B_s, \qquad t \leq s \leq T, \]
where $(B_t)_{0 \leq t \leq T}$ is  a $(\F_t)_{0 \leq t \leq T}$-Brownian motion and $\sigma \in \R^{d \times d}$ is a given matrix such that $\sigma \sigma^\top = a$.
Since $\nabla U$ is Lipschitz, it is standard that this stochastic differential equation (SDE) has a pathwise-unique strong solution.

As usual in control theory \cite{fleming2006controlled,yong1999stochastic}, the classical solution of \eqref{eq:HJB} given by Proposition \ref{pro:Well-Posed} can be represented as
the value function
\begin{equation} \label{eq:Value}
\varphi_t (x) := \inf_{u, \, X^u_t = x} \E \bigg[ \int_t^T \frac{1}{2} \vert u_s \vert^2 + V_s ( X^u_s) \, \d s + g ( X^u_T ) \bigg]. 
\end{equation} 
Moreover, the optimally controlled process $( X_s )_{t \leq s \leq T}$ is obtained as the solution of the stochastic differential equation
\begin{equation} \label{eq:OptiProc}
\d X_s = - \sigma \sigma^\top [ \nabla U_s (X_s ) + \nabla \varphi_s ( X_s ) ] \d s + \sigma \d B_s, \quad X_t = x. 
\end{equation} 
corresponding to the feed-back control $u_s = - \sigma^\top \nabla \varphi_s ( X_s )$.
The above SDE is well-posed because $\nabla \varphi$ is locally Lipschitz from Proposition \ref{pro:Well-Posed}, and \ref{ass:Coerciv} guarantees that either $\nabla \varphi$ has linear growth or the solution $\varphi_s$ of \eqref{eq:HJB} can be used as a Lyapunov function to ensure non-explosion -- Ito's formula and \eqref{eq:HJB} indeed entail that $\tfrac{\d}{\d s} \E [ \varphi_s ( X_s ) ] \leq 0$.

Introducing the adjoint variable $Y_s = \nabla \varphi_s ( X_s)$, Ito's formula combined with \eqref{eq:HJB} yields the forward-backward system (or FBSDE)
\begin{equation} \label{eq:Pontryagin}
\begin{cases}
\d X_s = - \sigma \sigma^\top [\nabla U_s (X_s) + Y_s ] \d s + \sigma \d B_s, \quad &X_t=x,\\
\d Y_s = [ \nabla^2 U_s (X_s) \sigma \sigma^\top Y_s - \nabla V_s (X_s)] \d s +  \nabla^2 \varphi_s (X_s) \sigma \d B_s, \quad &Y_T = \nabla g(X_T),
\end{cases}
\end{equation}
which is known as the stochastic maximum principle or stochastic Pontryagin system \cite{yong1999stochastic}, see Remark  \ref{rem:FBSDE} for further justifications.

The key idea for the proof of Theorem \ref{thm:PropHJB} in Section \ref{sec:ProofPropHJB} is to couple the optimal process $(X_s)_{t \leq s \leq T}$ starting from $x$ to a suitable choice of the optimal process $(\hat{X}_s)_{t \leq s \leq T}$ starting from $\hat{x}$, and to compute the time-evolution of $\vert X_s - \hat{X}_s \vert^{-1} \langle \nabla [ \varphi_s + U_s ] ( X_s ) - \nabla [ \varphi_s + U_s ] ( \hat{X}_s ) , [ X_s - \hat{X}_s ] \rangle$. 
Indeed, this quantity is controlled at time $s = T$ by $\kappa_{g + U_T}$, whereas it equals $\vert x - \hat{x} \vert^{-1} \langle \nabla [ \varphi_t + U_t ] ( x ) - \nabla [ \varphi_t + U_t ] ( \hat{x} ) , [ x - \hat{x} ] \rangle$ as desired at time $s = t$.
Our degree of freedom in constructing $(\hat{X}_s)_{t \leq s \leq T}$ is precisely the choice of the Brownian motion $( \hat{B}_s)_{t \leq s \leq T}$ such that
\[ \d \hat{X}_s = - \sigma \sigma^\top [\nabla U_t ( \hat{X}_s ) + \nabla \varphi_s ( \hat{X}_s ) ] \d s + \sigma \d \hat{B}_s, \quad \hat{X}_t=x, \]
which does not need to be $(B_s)_{t \leq s \leq T}$.
Leveraging the invariance of Brownian motion under rotations, the key point of coupling by reflection \cite{lindvall1986coupling,chen1989coupling,eberle2016reflection} is to choose for $( \hat{B}_s)_{t \leq s \leq T}$ a suitably rotated version of $( B_s)_{t \leq s \leq T}$ that helps $X_t$ and $\hat{X}_t$ to get closer on average.
Under adequate assumptions on the semiconvexity profiles, this can be quantified using a concave modification of the Euclidean distance \cite{eberle2016reflection}.
This method was successfully used in \cite{eberle2016reflection} and many following works to prove exponential contraction in Wasserstein-$1$ for diffusion processes, and in \cite{conforti2023coupling} to compute turnpike estimates for controlled diffusions.
In Section \ref{sec:ProofPropHJB}, we leverage coupling by reflection to prove weak semiconvexity estimates for \eqref{eq:HJB}, following the method developed in \cite{conforti2024weak} for the case $U \equiv V \equiv 0$.

\begin{rem}[PDE interpretation of coupling]
Denoting by $\rho_s$ the joint law of $(X_s,\hat{X}_s)$, $\rho_s$ can be characterized as the solution of a Fokker-Planck PDE on $\R^d \times \R^d$ with Dirac mass initial condition $\rho_t = \delta_{(x,\hat{x})}$. 
Our computations in Section \ref{sec:ProofPropHJB} then amount to lower bounding
\[ \frac{\d}{\d s} \int_{\R^d \times \R^d} \vert y - \hat{y} \vert^{-1} \langle \nabla [ \varphi_s + U_s ] ( y ) - \nabla [ \varphi_s + U_s ] ( \hat{y} ) , [ y - \hat{y} ] \rangle \d \rho_s ( \d y , \d \hat{y} ). \]
For an introduction to coupling proofs by PDE methods we refer to \cite{fournier2019monge}, which develops the case of synchronous coupling.
The needed adaptations for reflection coupling are described in \cite[Appendix A]{porretta2013global}, see also the forthcoming work \cite{bocchiSecondOrderEstimates}.
\end{rem}

\begin{rem}[FBSDE system] \label{rem:FBSDE}
The derivation of the Pontryagin system \eqref{eq:Pontryagin} is standard in control theory \cite{yong1999stochastic}.
We refer to \cite[Proof of Theorem 2.1, Equation (20)]{conforti2024weak} for an elementary derivation from \eqref{eq:HJB} in a similar setting.
This derivation requires to differentiate \eqref{eq:HJB} with respect to $x$, involving a third order term $\nabla \Delta\varphi_t$ which is cancelled in the derivation of \eqref{eq:Pontryagin}.
This difficulty is classically solved by smoothing and truncating the coefficients, thus producing a $\C^\infty$ solution $\varphi$ \cite[Corollary 5.6]{chaintron2024gibbs}, and then taking the limit using the a priori estimates from Proposition \ref{pro:Well-Posed} -- see e.g. \cite[Proposition 5.2]{chaintron2024gibbs} for a similar setting.
We chose to build the solution of the FBSDE system directly from the PDE \eqref{eq:HJB}, thus requiring a priori estimates given in Proposition \ref{pro:Well-Posed}.
The proof of Theorem \ref{thm:PropHJB} below would work the same way if \eqref{eq:Pontryagin} were obtained using usual BSDE arguments \cite{yong1999stochastic}.  
\end{rem}

\begin{rem}[Time-reversal of diffusion processes]
In the dual setting $(U,V) = (-U,-\Delta U)$ of \eqref{eq:DualFP}, we notice that $\varphi_t = - \log \mu_{T-t}$ and the optimal process \eqref{eq:OptiProc} satisfies
\begin{equation} \label{eq:timeRev}
\d X_t = a \nabla U_t ( X_t ) \d t + a \nabla \log \mu_{T-t} ( X_t ) \d t + \sigma \d W_t, 
\end{equation} 
for a suitable Brownian motion $(W_t)_{0 \leq t \leq T}$.
From \cite{haussmann1986time}, this process is precisely the time-reversal of the diffusion process $(X^0_t)_{0 \leq t \leq T}$ for $u \equiv 0$.
\end{rem}

\section{Proof of results on HJB equations} \label{sec:ProofPropHJB}

This section is devoted to the proof of the results in Section \ref{ssec:HJB} on HJB equations, starting from the framework introduced in Section \ref{ssec:ControlSto}.

\begin{proof}[Proof of Proposition \ref{pro:Well-Posed}]
Let us define the function $\varphi_t$ through the stochastic control representation \eqref{eq:Value}.
If we consider the control $u \equiv 0$, the Lipschitz assumption on $\nabla U$ classically implies that
\[ \forall p >1, \, \exists C^p_T > 0: \forall (t,x) \in [0,T] \times \R^d, \qquad \E \big[ \sup_{t \leq s \leq T} \vert X^0_s \vert^p \big] \leq C^p_T [ 1+ \vert x \vert^p ], \]
where we implicitly assumed that $X^0_t = x$.
From the polynomial growth assumption on coefficients \ref{ass:coeff}, we deduce that $\varphi$ has polynomial growth.
We can then reduce the minimisation in \eqref{eq:Value} to controls $u$ such that
\begin{equation} \label{eq:polyMoments}
\int_t^T \E [ V_s ( X^u_s ) ] \d s \leq C [ 1 + \vert x \vert^p], \qquad \E [ h ( X^u_T ) ] \leq C [ 1 + \vert x \vert^p],
\end{equation} 
for some $C,p >0$ independent of $(t,x)$ -- where once again $X^u_t = x$.
For $x, \hat{x} \in \R^d$, the Lipschitz assumption on $\nabla U$ and Ito's formula  classically imply that
\begin{equation} \label{eq:LipUnif}
\forall p >1, \, \exists K^p_T > 0: \forall (t,x,\hat{x}) \in [0,T] \times \R^d \times \R^d, \quad \quad \E \big[ \sup_{t \leq s \leq T} \vert X^u_s - \hat{X}^u_s \vert^p \big] \leq K^p_T \vert x - \hat{x} \vert^p, 
\end{equation} 
where $(X^u_t,\hat{X}^u_t) = (x,\hat{x})$, both processes being driven by the same Brownian motion.
Under \eqref{eq:Doubling},
we then write
\begin{align*} 
\E [ \vert V_s ( X^u_s) - V_s ( \hat{X}^u_s) \vert ] &= \E \bigg[ \int_0^1 \nabla V_s ( (1-r) \hat{X}^u_s + r X^u_s ) \cdot ( X^u_s - \hat{X}^u_s ) \d r \bigg] \\
&\leq C [ 1 + \E^{1/\alpha} [ V_s ( X^u_s) ] + \E^{1/\alpha} [ V_s ( \hat{X}^u_s) ] ] \E^{1/\alpha'} [ \vert X^u_s - \hat{X}^u_s \vert^{\alpha'} ], 
\end{align*}
using Hölder's inequality with the conjugate exponent $\alpha'$ of $\alpha$ given by \eqref{eq:Doubling} -- with $\alpha' = 1$ if $\nabla V$ is bounded.
The same reasoning holds with $h( X^u_T )$ in place of $V_s ( X^u_s)$.
Combining this with \eqref{eq:polyMoments}-\eqref{eq:LipUnif}, we obtain a local Lipschitz bound on $\vert \varphi_t ( x) - \varphi ( \hat{x} ) \vert$, and we then deduce that $\vert \varphi_t (x) \vert + \vert \nabla \varphi_t (x) \vert \leq C [ 1 + \vert x \vert^p ] $ for some $C,p >0$ independent of $(t,x)$. 

Let us now consider the functions, for $(s,x,u) \in [0,T] \times \R^d \times \R^d$,
\begin{equation} \label{eq:DefReg}
b_s (x,u) := -\sigma \sigma^\top \nabla U_s (x
) + \sigma {\sf P}_x u, \qquad f_s (x,u) = \tfrac{1}{2} \vert {\sf P}_x u \vert^2 + V_s (x), 
\end{equation} 
where ${\sf P}_x u$ is the projection of $u$ on the closed ball $\overline{B}(0,C[1+\vert x \vert^p])$ in $\R^d$.
We notice that $b$, $f$ have polynomial growth in $x$ uniformly in $u$, and $b$ (resp. $f$) satisfies the same assumptions as $\nabla U$ (resp. $V$). We then define
\begin{equation}  
\psi_t (x) := \inf_{u, \, Y^u_t = x} \E \int_t^T f_s ( Y^u_s, u_s )  + V_s ( Y^u_s) \, \d s + h ( Y^u_T ), 
\end{equation} 
for the controlled dynamics
\[ \d Y^u_s = b_s ( Y^u_s, u_s) \d s + \sigma \d B_s, \qquad t \leq s \leq T. \]
The above reasoning shows that $\vert \psi \vert + \vert \nabla \psi \vert$ enjoy the same polynomial bound as $\vert \varphi \vert + \vert \nabla \varphi \vert$. 

Our uniform in $u$ assumptions now enter the framework of \cite[Chapters 3-5]{krylov2008controlled}.
\cite[Chapter 4,Section 7,Theorem 4]{krylov2008controlled} first  shows that $\psi$ is $W^{1,2}_{\mathrm{loc}}$ with polynomial growth derivatives.
In particular, $x \mapsto \psi_t (x)$ is $\C^1$ with locally Lipschitz derivatives.
\cite[Chapter 4,Section 7,Theorem 2]{krylov2008controlled} further shows that $\psi$ is a.e. solution of the HJB equation
\begin{equation} \label{eq:ApproxHJB}
\partial_t \psi_t + \inf_{u \in \overline{B}(0,C[1+\vert x \vert^p])} b_t (x,u) \cdot \nabla \psi_t + f_t (x,u) + \frac{1}{2} \mathrm{Tr} [ a \nabla^2 \psi_t ]= 0. 
\end{equation}
\cite[Chapter 5,Section 3, Theorem 14]{krylov2008controlled} eventually shows that $\psi$ is the unique polynomial growth solution of \eqref{eq:ApproxHJB} in $\C \cap W^{1,2}_{\mathrm{loc}}$.
Using \eqref{eq:DefReg} and the fact that $\nabla \psi$ satisfies the same a priori bound as $\nabla \varphi$, 
we deduce that the infimum in \eqref{eq:ApproxHJB} is realised by $u = \sigma^\top \nabla \psi_t$, so that $\psi_t$ actually satisfies the non-truncated equation \eqref{eq:HJB}.
Freezing the non-linearity, we can locally apply the Schauder regularity theory for linear parabolic PDE to deduce that $\psi$ is actually $\C^{1,2}$ with Hölder derivatives -- see e.g. \cite[Corollary 4.10 and Exercise 4.3] 
{lieberman1996second}.
Since $\psi$ is $\C^{1,2}$, the standard verification argument \cite{fleming2006controlled,yong1999stochastic} concludes the proof by showing that $\psi_t = \varphi_t$.
\end{proof}

\begin{proof}[Proof of Theorem \ref{thm:PropHJB}]
Let $\sigma_A$ be the positive definite symmetric square-root of $A\sigma(A\sigma)^\top$. Then $\sigma := A^{-1}\sigma_A$ satisfies $\sigma \sigma^\top = a$, and we fix this choice of $\sigma$ in the control formulation \eqref{eq:Value}.
Let us fix $x \neq \hat{x}$ in $\R^d$, and a function $f$ satisfying \eqref{eq:fIneq}. Assuming that
$\kappa^{A}_{U_T + g} (r,e) \geq r^{-1} f_T (r,e)$ for every $(r,e)$, it is sufficient to show that 
\[ \langle A a [ \nabla U_0 ( x ) + \nabla \varphi_0 ( x ) - \nabla U_0 ( x ) - \nabla \varphi_0 ( x ) ], A [ x - \hat{x} ] \rangle \geq \vert A(x - \hat{x}) \vert f_0 \big( \vert A ( x - \hat{x}) \vert , \tfrac{ A ( x - \hat{x} )}{\vert A ( x - \hat{x} ) \vert} \big). \]
We thus consider the Pontryagin system \eqref{eq:Pontryagin} starting from $x$ at time $0$.
Rather than directly working on $(X_t,Y_t)$, recalling that $Y_t = \nabla \varphi_t (X_t)$, we are going to build a coupling by reflection for $(X^A_t,Y^A_t) := (A X_t,A Y_t)$, which is an equivalent task.

Let $(B^1_t)_{t\geq0},~(B^2_t)_{t\geq0}$ be two independent Brownian motions, and $\tilde{\sigma}_A$ be the positive semi-definite square-root of $\sigma_A\sigma_A^\top - \overline{\sigma}_A^2 \Id$, recalling that $\overline{\sigma}_A >0$ is the lowest eigenvalue of $Aa A^\top$.
To adapt the reflection to the geometry given by $A$, we choose the Brownian motion driving $( X^A_t )_{0 \leq t \leq T}$ as being
\[ B_t := \sigma_A^{-1} [ \overline{\sigma}_A B^1_t + \tilde{\sigma}_A B^2_t ], \]
which indeed satisfies $\mathrm{Cov}(B_t) = t \Id$.
The coupled process $( \hat{X}^A_t, \hat{Y}^A_t)$ is defined by 
\begin{equation} \label{eq:coupledSys}
\begin{cases}
\d \hat X^A_t = - A\sigma \sigma^\top [\nabla U_t (\hat X_t) + \hat Y_t ] \d t + \sigma_A \d \hat{B}_t, \quad &\hat X^A_0=A \hat x,\\
\d \hat Y^A_t = [ A\nabla^2 U_t (\hat X_t) \sigma \sigma^\top \hat Y_t - A\nabla V(\hat X_t)] \d t + A \nabla^2 \varphi_t (\hat X_t) \sigma \d \hat B_t, \quad &\hat Y^A_T = A\nabla h(\hat X_T),
\end{cases}
\end{equation}
where $\hat{X}_t := A^{-1} X^A_t$, $\hat{Y}_t := \nabla \varphi_t ( \hat{X}_t )$, $\hat{Y}^A_t := A \hat{Y}_t$, $\hat{B}_t := \sigma_A^{-1} [ \overline{\sigma}_A \hat{B}^1_t + \tilde{\sigma}_A B^2_t$ ], and 
\begin{align}
\d \hat{B}^1_t &:= \big[ \mathrm{Id}-2\rme^A_t (\rme^A_t)^\top \big] \De B^1_t,  \qquad \rme^A_t := \frac{X^A_t - \hat{X}^A_t}{\vert X^A_t - \hat{X}^A_t \vert},
\end{align}
while $t < \tau$ where $\tau$ is the stopping time
\[ \tau := \inf \{ t \in [0,T], \; X^A_t \neq \hat{X}^A_t \}, \]
and $\hat{B}^1_t = B^1_t$ for $t \geq \tau$, so that $X^A_t = \hat{X}^A_t$ for $t \geq \tau$.
Following \cite{eberle2016reflection}, the diffusion process $(X^A_t,\hat{X}^A_t)_{0 \leq t \leq T}$ is well-defined because $\nabla \varphi$ is locally-Lipschitz from Proposition \ref{pro:Well-Posed} and \ref{ass:Coerciv} guarantees non-explosion as in Section \ref{ssec:ControlSto}. The fact that $( \hat{B}_t )_{0 \leq t \leq T}$ is a Brownian motion with covariance matrix $\Id$ is given by Lévy's characterisation.
We then define $r^A_t := \vert X^A_t- \hat{X}^A_t \vert$, as well as 
\[
\delta_t (X_t,\hat{X}_t) := - \sigma \sigma^\top [ \nabla U_t (X_t) + Y_t - \nabla U_t (\hat X_t) - \hat Y_t] , \quad
\Theta^A_t(X_t,\hat{X}_t) := -A \delta_t (X_t,\hat{X}_t) \cdot \rme^A_t. \]
We recall that $(X^A_t,Y^A_t)$ satisfies the same system \eqref{eq:coupledSys} with $B_t$ in place of $\hat{B}_t$.
Recalling that $\overline{\sigma}_A$ is a positive scalar, Lemma \ref{lem:angle} below gives that, for $t < \tau$,
\begin{equation} \label{eq:diffe_t}
\d [ X^A_t - \hat{X}^A_t ] = A\delta_t (X_t,\hat{X}_t) \d t + 2 \overline{\sigma}_A \rme^A_t (\rme^A_t)^\top \d B^1_t, \quad \d \rme_t^A = (r^A_t)^{-1} \pi_{\rme^A_t}^{\perp} [ A\delta_t(X_t,\hat{X}_t) ] \De t.
\end{equation}
Using It\^o's formula we further get
\begin{equation*}
\d \nabla U_t (X_t) = \big[ \nabla \partial_t U_t ( X_t) + \nabla^2 U_t (X_t)[-\sigma \sigma^\top\nabla U_t (X_t)- \sigma \sigma^\top Y_t] + \tfrac{1}{2} \nabla \mathrm{Tr}[\sigma \sigma^\top \nabla^2 U_t] \big] \d t + \d M_t,
\end{equation*}
where $(M_t)_{0 \leq t \leq T}$ denotes a generic local martingale, which may change from line to line. 
Recalling the definition \eqref{eq:ReciPot} of the reciprocal potential, we obtain \begin{equation} \label{eq:GradSimplif}
\d [ \nabla U_t (X_t)+ Y_t ] = - \nabla [ {\A_U}_t (X_t) + V_t ( X_t ) ] \d t +\d M_t, 
\end{equation}
and then, still for $t < \tau$,
\[ 
\d A\delta_t (X_t,\hat{X}_t) = - A\sigma \sigma^\top [ \nabla [ {\A_U}_t + V_t ] (X_t)- \nabla [ {\A_U}_t + V_t ] (\hat{X}_t) ] \d t +\d M_t. \]
We can finally compute, 
recalling that $\Theta^A_t(X_t,\hat{X}_t) := -A\delta_t (X_t,\hat{X}_t) \cdot \rme^A_t$ and that $\d \rme_t$ has no martingale argument from \eqref{eq:diffe_t}, 
\begin{align} \label{eq:DiffTheta}
\De \Theta^A_t(X_t,\hat{X}_t) &= -\rme^A_t \d (A \delta_t (X_t,\hat{X}_t)) - A\delta_t (X_t,\hat{X}_t) \d \rme^A_t + \De M_t \\
&= - A\sigma \sigma^\top [ \nabla [ {\A_U}_t + V_t ] (X_t)- \nabla [ {\A_U}_t + V_t ] (\hat{X}_t) ] \cdot \rme^A_t \d t 
\\
&\qquad\quad+ ( r^A_t )^{-1} |\pi_{(\rme^A_t)^{\perp}}[ A\delta_t (X_t,\hat{X}_t)]|^2 \De t + \d M_t, \\
&\leq - \kappa^{A}_{\cA_t + V_t}(r^A_t,\rme^A_t)r^A_t \d t + ( r^A_t )^{-1} |\pi_{\rme^A_t}^{\perp}[ A\delta(X_t,\hat{X}_t)]|^2 \De t + \d M_t,
\end{align}
recalling the definition \eqref{eq:AnProfile} of $\kappa^{A}$.
For $t < \tau$, Lemma \ref{lem:angle} below yields
\[
\d r^A_t = A\delta_t (X_t,\hat{X}_t) \cdot \rme^A_t \De t + 2\overline{\sigma}_A \De W^A_t = -\Theta^A_t(X_t,\hat{X}_t) \De t + 2\overline{\sigma}_A \De W^A_t, \]
for the one-dimensional Brownian motion given by $\De W^A_t := (\rme^A_t)^\top \d B^1_t$.
Ito's rule and the inequality \eqref{eq:fIneq} then imply, using \eqref{eq:diffe_t} for the second equality,
\begin{align}
\d f_t (r^A_t,\rme^A_t) &= [ \partial_t f_t (r^A_t,\rme^A_t) + 2 \overline{\sigma}_A^2 \partial_{rr} f_t (r^A_t,\rme^A_t) ] \d t + \nabla_e f_t (r^A_t,\rme^A_t) \cdot \d \rme^A_t + \partial_r f_t (r^A_t,\rme^A_t) \d r^A_t \\
&= [ \partial_t f_t (r^A_t,\rme^A_t) - \partial_r f_t (r^A_t,\rme^A_t) \Theta^A_t(X_t,\hat{X}_t ) + 2 \overline{\sigma}_A^2 \partial_{rr} f_t (r^A_t,\rme^A_t) ] \d t \\
&\qquad \quad + (r^A_t)^{-1} \nabla_e f_t (r^A_t,\rme^A_t) \cdot \pi_{\rme^A_t}^{\perp}[ A\delta_t(X_t,\hat{X}_t) ] \d t + \d M_t.
\end{align}
Using the identity $-a^2 + a b \leq b^2/4$ and the orthogonality of the projection $\pi^\perp_{\rme^A_t}$, 
\begin{multline*}
-( r^A_t )^{-1} |\pi_{\rme^A_t}^{\perp} [- A\delta(X_t,\hat{X}_t)]|^2 - (r^A_t)^{-1} \nabla_e f_t (r^A_t,\rme^A_t) \cdot \pi_{\rme^A_t}^{\perp}[ A\delta_t(X_t,\hat{X}_t) ] \\
\leq (4 r^A_t )^{-1} \vert \pi_{\rme^A_t}^\perp [ \nabla_e f_t ( r^A_t, \rme^A_t ) ] \vert^2,
\end{multline*}
so that gathering terms yields for $t < \tau$,
\begin{align*}
\d [\Theta^A_t(X_t,\hat{X}_t) -f_t(r^A_t,\rme^A_t)] &\leq -[ \partial_t f_t (r^A_t,\rme^A_t) - \partial_r f_t (r^A_t,\rme^A_t) \Theta^A_t(X_t,\hat{X}_t ) \d t + 2 \overline{\sigma}_A^2 \partial_{rr} f_t (r^A_t,\rme^A_t) ] \d t \\
&\qquad \quad- \kappa^{A}_{\cA_t + V_t}(r^A_t)r^A_t + (4 r^A_t )^{-1} \vert \pi_{\rme^A_t}^\perp [ \nabla_e f_t ( r^A_t, \rme^A_t ) ] \vert^2 \d t + \d M_t \\ 
&\leq \partial_r f_t(r^A_t,\rme^A_t)[\Theta^A_t(X_t,\hat{X}_t)-f_t(r^A_t,\rme^A_t)]\d t + \d M_t,
\end{align*} 
using the differential inequality \eqref{eq:fIneq} satisfied by $f_t$.
For $t \geq \tau$, we have $r^A_t = 0 = \Theta^A_t(X_t,\hat{X}_t)$, and we define $f_\tau ( r^A_\tau, \rme^A_\tau ) := \limsup_{t \uparrow \tau} f_t ( r^A_t, \rme^A_t )$. 
Setting
\[
\Gamma_t := \exp \bigg[ -\int_0^t \partial_r f_s ( r^A_s, \rme^A_s ) \d s \bigg][\Theta^A_t(X_t,\hat{X}_t)-f_t(r^A_t,\rme^A_t)], \]
we deduce that $\d \Gamma_t \leq \d M_t$ and then $\E [ \Gamma_0 ] \geq \E [ \Gamma_{T \wedge \tau} ]$.
If $T \wedge \tau = T$, we have $\Gamma_{T \wedge \tau} \geq 0$ from our assumption on the profile $f_T$ at the terminal time.
Otherwise, $\Gamma_{T \wedge \tau} \geq 0$ is guaranteed by the second condition in \eqref{eq:fIneq}.
We thus deduce that $\E [ \Gamma_0 ] \geq 0$ as desired. 

Since we only dealt with lower bound on convexity profiles, we can obviously replace $\kappa^{A}_{\A_t + V_t}$ in \eqref{eq:fIneq} by any lower bound on $\kappa^{A}_{\A_t + V_t}$.
\end{proof}

\begin{lemma}[Deterministic angle] \label{lem:angle}
The following hold for $t < \tau$,
\[ \d r^A_t = A\delta_t (X_t,\hat{X}_t) \cdot \rme^A_t \De t + 2\overline{\sigma}_A \De W^A_t, \qquad \d \rme_t^A = (r^A_t)^{-1} \pi_{\rme^A_t}^{\perp} [ A\delta_t(X_t,\hat{X}_t) ] \De t. \]
for the one-dimensional Brownian motion given by $\De W^A_t := (\rme^A_t)^\top \d B^1_t$, and  $\pi^\perp_{\rme^A_t}$ the orthogonal projection on the hyperlane with normal vector $\rme^A_t$. 
\end{lemma}

\begin{proof}
From \eqref{eq:coupledSys}, we have
\[ \d [ X^A_t - \hat{X}^A_t ] = A\delta_t (X_t,\hat{X}_t) \d t + 2 \overline{\sigma}_A \rme^A_t (\rme^A_t)^\top \d B^1_t. \]
Recalling that $r^A_t := \vert X^A_t- \hat{X}^A_t \vert$, the first identity is then a direct consequence of Ito's formula, already detailed in \cite[Equation (60)]{eberle2016reflection} or \cite[Lemma 2.1]{conforti2024weak}. 
For the second one, we use Ito's formula to write
\begin{align}
    \De ( r^A_t )^{-1} 
    &= -( r^A_t )^{-2} \De r^A_t + (r^A_t)^{-3} \De [r^A_\cdot]_t \\
    &= [-(r^A_t)^{-2}A\delta(X_t,\hat{X}_t) \cdot \rme^A_t + 4\overline{\sigma}_A^2(r^A_t)^{-3} ] \De t - 2(r^A_t)^{-2} \overline{\sigma}_A \De W^A_t,
\end{align}
and then
\begin{align}
    \De \rme^A_t 
    &= ( r^A_t )^{-1} \De (X^A_t - \hat{X}^A_t) +  (X^A_t - \hat{X}^A_t)  \De ( r^A_t )^{-1} + \De [X^A_\cdot - \hat{X}^A_\cdot,( r^A_\cdot )^{-1}]_t \\
    &= ( r^A_t )^{-1}A\delta(X_t,\hat{X}_t) \De t + 2 ( r^A_t )^{-1}\overline{\sigma}_A \rme^A_t \De W^A_t - 4\rme^A_t(r^A_t)^{-2} \overline{\sigma}_A^2 \De t \\
    &\!\!\!\!+ (X^A_t - \hat{X}^A_t) [-(r^A_t)^{-2}A\delta(X_t,\hat{X}_t) \cdot \rme^A_t + 4\overline{\sigma}_A^2(r^A_t)^{-3} ] \d t - 2 (r^A_t)^{-2} \overline{\sigma}_A (X^A_t - \hat{X}^A_t) \d W^A_t \\
    &= ( r^A_t )^{-1} [ A\delta(X_t,\hat{X}_t) - (X^A_t - \hat{X}^A_t)(r^A_t)^{-1}A\delta(X_t,\hat{X}_t) \cdot \rme^A_t ] \De t \\
    &= -( r^A_t )^{-1} [- A\delta(X_t,\hat{X}_t) + (A\delta(X_t,\hat{X}_t) \cdot \rme^A_t)\rme^A_t ] \De t \\
    &= -( r^A_t )^{-1} \pi_{(\rme^A_t)^{\perp}}[- A\delta(X_t,\hat{X}_t)] \De t.
\end{align}
\end{proof}

\begin{rem}[Gradient structure]
A tempting generalization would replace $-\nabla U ( X_t)$ by a general Lipschitz drift term $b ( X_t )$.
However, the simplification \eqref{eq:GradSimplif} that makes the reciprocal potential appear no more holds in this setting.
\end{rem}

\begin{proof}[Proof of Theorem \ref{thm:ProPJoint}]
Let us first estimate the convexity profile with respect to $z$. To do so,
we exactly follow the proof of Theorem \ref{thm:PropHJB}, except that we replace the quantity $(Y_t,\hat{Y}_t)$ -- which was corresponding to $(\nabla \varphi_t ( X_t), \nabla \varphi_t ( \hat{X}_t))$ -- by $(\nabla_z \varphi^x ( X_t ),\nabla_z \varphi^{\hat{x}} ( \hat{X}_t ))$.
The result follows without any change in the proof.

To estimate the convexity with respect to $x$, let us introduce 
\[ \delta^A_t := \vert A (x - \hat{x}) \vert^{-1} \langle A a [ \nabla_x \varphi^x ( X_t) - \nabla_{x} \varphi^{\hat{x}} ( \hat{X}_t ) ], A [x - \hat{x}] \rangle. \]
For every $1 \leq i \leq d$, differentiating \eqref{eq:HJjoint} with respect to $x$ yields
\begin{equation}
\partial_t \partial_{x_i} \varphi^x_t - a [ \nabla_z U_t + \nabla_z \varphi^x_t ]  \cdot \nabla_z \partial_{x_i} \varphi^x_t + \frac{1}{2} \mathrm{Tr} [ a \nabla^2_z \partial_{x_i} \varphi^x_t ] = 0.
\end{equation} 
Since the initial datum $h$ is $\C^2$ with respect to $x$ and $\partial_{x_i} \varphi^x_t$ can be defined from the above linear parabolic PDE -- see e.g. \cite{rubio2011existence} for well-posedness --, the differentiability of $\varphi^x_t$ with respect to $x$ follows from standard arguments.

Ito's formula then yields $\d \nabla_{x} \varphi^x_t ( X_t ) = \sigma \nabla^2_{xz} \varphi^x_t ( X_t ) \d B_t$, and the analogous holds for $\d \nabla_{x} \varphi^{\hat{x}}_t ( \hat{X}_t )$.  
As a consequence, $( \delta^A_t )_{0 \leq t \leq T}$ is a local martingale.
Using Ito's formula as we did for $f_t$ in the proof of Theorem \ref{thm:PropHJB}, for $t < \tau$,
\begin{align}
\d g^{x,\hat{x}}_t (r^A_t ) &= [ \partial_t g^{x,\hat{x}}_t (r^A_t ) - \partial_r g^{x,\hat{x}}_t (r^A_t ) \Theta^A_t(X_t,\hat{X}_t ) + 2 \overline{\sigma}_A^2 \partial_{rr} g^{x,\hat{x}}_t (r^A_t ) ]\d t  + \d M_t, \\
&= [ \partial_t g^{x,\hat{x}}_t (r^A_t ) - g^{x,\hat{x}}_t (r^A_t ) \partial_r g^{x,\hat{x}}_t (r^A_t ) + 2 \overline{\sigma}_A^2 \partial_{rr} g^{x,\hat{x}}_t (r^A_t ) ]\d t  + \d M_t \\
&\quad\quad+ \partial_r g^{x,\hat{x}}_t (r^A_t ) [ g^{x,\hat{x}}_t (r^A_t ) - \Theta^A_t(X_t,\hat{X}_t ) ] \d t.
\end{align}
Since we assumed $f^{x,\hat{x}}_t (r^A_t, \rme^A_t ) \geq g^{x,\hat{x}}_t (r^A_t )$, we have $g^{x,\hat{x}}_t (r^A_t ) - \Theta^A_t(X_t,\hat{X}_t ) \leq 0$ using the convexity estimate with respect to $z$.
Since $\partial_r g^{x,\hat{x}}_t \leq 0$, we can deduce from \eqref{eq:gIneq} that $t \mapsto \E [ g_{t \wedge \tau} ( r^A_{t \wedge \tau} )]$ is non-decreasing, defining $g_{t \wedge \tau} ( r^A_{t \wedge \tau} ) := \limsup_{t \uparrow \tau} g_t ( r^A_t )$ as previously.    
Hence,
\[ \delta^A_0 - g^{x,\hat{x}}_0 (r^A_0 ) = \E [ \delta^A_0 - g^{x,\hat{x}}_0 (r^A_0 ) ] \geq \E [ \delta^A_T - g^{x,\hat{x}}_{T \wedge \tau} (r^A_{T \wedge \tau} ) ] \geq 0, \]
using our convexity assumption with respect to $x$ and our assumption $\limsup_{r \downarrow 0} g_t (r) \leq 0$. This concludes the proof.
\end{proof}

\begin{proof}[Proof of Corollary \ref{cor:Two-sided}]
Let us leverage the convexity property from Corollary \ref{cor:Invar} to get Hessian estimates.
For $x \in \R^d$, let us consider the stochastic process solution of 
\[ \d X_s = - \nabla[U + \varphi^\varepsilon_s] ( X_s ) \d s + \d B_s, \quad t \leq s \leq T, \quad X_t = x, \]
where $(B_s)_{t \leq s \leq T}$ is a Brownian motion.
Adapting the computation \eqref{eq:Pontryagin}, Ito's formula yields
\begin{equation} \label{eq:UnifHessian}
\nabla [ U + \varphi^\varepsilon_t ] (x) = \E \bigg[ \nabla [ U + \varphi^\varepsilon_T ] ( X_T ) + \int_t^T \nabla [ \A_U + V ] (X_s) \d s \bigg]. 
\end{equation}
Let $( \hat{X}_s )_{t \leq s \leq T}$ be the analogous process starting from $\hat{x}$ and coupled by reflection with $(X_s)_{t \leq s \leq T}$ -- we refer to the proof of Theorem \ref{thm:PropHJB} in Section \ref{sec:ProofPropHJB} for a detailed presentation and well-posedness for this coupling.
We have $\kappa_{U + \varphi_t} (r) \geq r^{-1} \inf_e f(r,e)$
with $\int_0^1 \inf_e f^-(r,e) \d r < + \infty$.
From \cite[Theorem 3.4]{priola2006gradient}, or \cite[Proposition 2.2-(ii)]{conforti2023coupling} in a setting close to ours, there exist $\lambda, C >0$ such that  
\[ \forall s \in [t,T], \quad \E [ \vert X_s - \hat{X}_s \vert ] \leq C e^{\lambda (s-t)} \vert x - \hat{x} \vert. \]
Since $\nabla U$, $\nabla [ \A_U + V]$ and $\nabla \varphi^\varepsilon_T = - \nabla \log \phi^\varepsilon$ are Lipschitz-continuous, \eqref{eq:UnifHessian} implies that $\vert \nabla \varphi_t ( x ) - \nabla \varphi_t ( \hat{x} ) \vert \leq K \vert x - \hat{x} \vert$ for a constant $K >0$.
If furthermore $\liminf_{r \rightarrow +\infty} r^{-1} \inf_e f(r,e) > 0$, we can take $\lambda < 0$ so that $K$ is independent of $(t,T)$, using \cite[Theorem 1]{eberle2016reflection}. 
\end{proof}

\begin{rem}[Bounded gradients]
Since reflection coupling also allows for total variation estimates, the above proof still works if $\nabla h$ and $\nabla [ \A_U + V]$ are only bounded functions.    
\end{rem}

\appendix

\section{Computation of explicit profiles} \label{app:Explicit}

This appendix is devoted to proofs that were omitted in Section \ref{ssec:ExistingLit}.

\subsection{Semiconvexity generation} \label{sec:appSC}

In the setting of Lemma \ref{lem:time-dependent-prof-reg}, we assume that $\kappa_t \equiv 0$.
The solution $F_t$ then simplifies to
\begin{multline} \label{eq:conv_kernel_conv-prof}
F_t(r) = \int_{-\infty}^{+\infty} \exp\bigg[-\frac{1}{4\overline{\sigma}_A^2} \int_0^{|x|} f_T(u) u\bigg] p_{T-t}(x,r) \d x \\ 
=\frac{1}{\sqrt{8\pi\overline{\sigma}_A^2(T-t)}}\int_0^\infty \exp\bigg[-\frac{1}{4\overline{\sigma}_A^2}\int_0^{x}f_T(u)\De u \bigg]\left[\exp\left(-\frac{(x-r)^2}{8\overline{\sigma}_A^2(T-t)}\right) + \exp\left(-\frac{(x+r)^2}{8\overline{\sigma}_A^2(T-t)} \right) \right] \De x.
\end{multline}
In the following, we need to specify the dependence on $(T,f_T)$ of the solution $f_t = f^{T,f_T}$ of \eqref{eq:exp_conv_prof} given by \eqref{eq:HopfCEx}.

\begin{lemma}\label{lem:conv_prof_unif_conv_bd}
If $\liminf_{r\rightarrow \infty} \kappa_{f_T}(r) \geq \alpha > -\frac{1}{T}$, then
\begin{equation}
f^{T,f_T}_t(r) = \frac{\alpha}{1+(T-t)\alpha}r - \frac{1}{1+(T-t)\alpha} f^{T_\alpha,f_T - \tfrac{\alpha}{2} \vert \cdot \vert^2}_{t_\alpha} \bigg(\frac{r}{1+ \alpha(T-t)}\bigg),
\end{equation}
where $T_\alpha := \frac{T}{1+ \alpha(T-t)}$ and $t_\alpha := \frac{t}{1+ \alpha(T-t)}$. 
\end{lemma}

\begin{proof}
This computation is a simple completing the squares. 
The assumed condition on $\alpha$ implies finiteness of $F_t (r)$ in \eqref{eq:conv_kernel_conv-prof}.
\end{proof}

Let us now compute lower bounds on eigenvalues of $\nabla^2 [ U_t + \varphi_t ]$.

\begin{lemma}\label{lem:prop_lower_eigenvalue}
If $\liminf_{r\rightarrow \infty}\kappa_{U_T + g}(r) \geq \alpha > -\frac{1}{T}$, then
\[ \lambda_{\mathrm{min}}[ \nabla^2 [ U_t+\varphi_t ]] \geq \frac{\alpha}{1+(T-t)\alpha} +\frac{1}{(T-t)(1+(T-t)\alpha)} - \frac{1}{4\overline{\sigma}^2_A (T-t)^2} \bbE[X_\alpha^2],
\]
where $X_\alpha$ is a random variable with density
$$\rho (\d x) \propto \1_{x \geq 0} \exp \bigg[ -\frac{1}{4\overline{\sigma}_A^2} \int_0^{x} f_T (u) - \alpha u \, \d u \bigg] \exp \bigg[ -\frac{x^2}{8\overline{\sigma}^2_A (T_\alpha - t_\alpha)} \bigg], $$
with $T_\alpha := \frac{T}{1+ \alpha(T-t)}$ and $t_\alpha := \frac{t}{1+ \alpha(T-t)}$.
\end{lemma}

\begin{proof}
Recall that from \eqref{eq:EigenBound}, we have
\begin{equation}
    \lambda_{\mathrm{min}}[ \nabla^2 [ U_t+\varphi_t ]] \geq \liminf_{r \downarrow 0}  \frac{f_t (r)}{r},
\end{equation}
where $f_t$ is the solution constructed in Lemma \ref{lem:time-dependent-prof-reg}.
We first deal with the case $\alpha =0$.
In order to bound from below the r.h.s., we compute from \eqref{eq:conv_kernel_conv-prof}
\begin{align}
\partial_r F_t(r) &= -\frac{1}{\sqrt{8\pi\overline{\sigma}^2_A (T-t)}}\int_0^\infty \exp\bigg[-\frac{1}{4\overline{\sigma}^2_A}\int_0^{x} f_T (u) \d u \bigg] \cdot \\
&\quad\quad \bigg[\frac{r-x}{4\overline{\sigma}^2_A (T-t)}\exp \bigg[ -\frac{(x-r)^2}{8\overline{\sigma}^2_A(T-t)} \bigg] + \frac{r+x}{4\overline{\sigma}^2_A (T-t)}\exp\bigg[ -\frac{(x+r)^2}{8\overline{\sigma}^2_A (T-t)} \bigg] \bigg] \d x \\
&= -\frac{r}{4\overline{\sigma}^2_A (T-t)} F_t(r) + \frac{1}{\sqrt{8\pi\overline{\sigma}^2_A (T-t)}} \int_0^\infty \exp \bigg[ -\frac{1}{4\overline{\sigma}^2_A}\int_0^{x}uk(u)\De u \bigg] \cdot \\
&\quad\quad\bigg[\frac{x}{4\overline{\sigma}^2_A(T-t)}\exp\bigg[-\frac{(x-r)^2}{8\overline{\sigma}^2_A (T-t)}\bigg] - \frac{x}{4\overline{\sigma}^2_A (T-t)} \exp\bigg[ -\frac{(x+r)^2}{8\overline{\sigma}^2_A(T-t)} \bigg] \bigg] \d x.
\end{align}
We then observe that
\begin{multline*}
\lim_{r \rightarrow 0}\frac{1}{r}\bigg[\frac{x}{4\overline{\sigma}^2_A(T-t)}\exp\bigg[-\frac{(x-r)^2}{8\overline{\sigma}^2_A(T-t)}\bigg] - \frac{x}{4\overline{\sigma}^2_A(T-t)}\exp \bigg[-\frac{(x+r)^2}{8\overline{\sigma}^2_A (T-t)} \bigg] \bigg] \\
= \frac{x^2}{8\overline{\sigma}^4_A(T-t)^2} \exp\bigg[-\frac{x^2}{8\overline{\sigma}^2_A(T-t)}\bigg].
\end{multline*}
Combining these computations with \eqref{eq:HopfCEx} yields
\begin{multline*}
\lim_{r \rightarrow 0} \frac{f_t(r)}{r} = \frac{1}{T-t} - \frac{1}{F_t(0)}\frac{1}{\sqrt{8\pi\overline{\sigma}^2
A (T-t)}}\int_0^\infty \exp \bigg[ -\frac{1}{4\overline{\sigma}^2_A} \int_0^{x}  f_T(u)\d u \bigg] \cdot \\
\frac{x^2}{2\overline{\sigma}^2_A (T-t)^2} \exp\bigg[ -\frac{x^2}{8\overline{\sigma}^2_A (T-t)}\bigg] \d x.
\end{multline*}
The lower bound on the eigenvalue for $\alpha =0$ then follows from \eqref{eq:EigenBound}.
The general case $\alpha >0$ can then be deduced using Lemma \ref{lem:conv_prof_unif_conv_bd}.
\end{proof}

Let us finally prove Proposition \ref{pro:SemiGene}.

\begin{proposition} \label{pro:AppSemiGene}
If $\kappa_{U_T + g}(r) \geq \alpha -\frac{L}{r}$ for $\alpha, L \geq 0$, then
\begin{align}
\lambda_{\mathrm{min}}&[ \nabla^2 (U_t +\varphi_t) ]\\
&\geq \frac{\alpha}{1+(T-t)\alpha}-\frac{1}{(1+(T-t)\alpha)^2} \bigg[\frac{L}{\sqrt{2\pi\overline{\sigma}^2_A (T_\alpha-t_\alpha)}} \exp\bigg[ -\frac{(T_\alpha-t_\alpha)L^2}{8\overline{\sigma}^2_A} \bigg] -\frac{L^2}{4\overline{\sigma}^2_A} \bigg] \\
&\geq \frac{\alpha}{1+(T-t)\alpha}-\frac{L}{\sqrt{2\pi\overline{\sigma}^2_A (T-t)(1+(T-t)\alpha)^3}} -\frac{L^2}{4\overline{\sigma}^2_A (1+(T-t)\alpha)^2} 
\end{align}
\end{proposition}

\begin{proof}
As previously, we write the proof for $\alpha =0$, the general case following from Lemma \ref{lem:conv_prof_unif_conv_bd}.
From the density
\[
\rho (\d x) 
\propto \1_{x \geq 0} \exp\bigg[-\frac{1}{4\overline{\sigma}^2_A}Lx \bigg] \exp\bigg[-\frac{x^2}{8\overline{\sigma}^2_A (T-t)}\bigg] \propto \1_{x \geq 0} \exp \bigg[ -\frac{(x-(T-t)L)^2}{8\overline{\sigma}^2_A (T-t)} \bigg], \]
we compute
\begin{align}
&\int_0^\infty x^2 \exp \bigg[-\frac{(x-(T-t)L)^2}{8\overline{\sigma}^2_A (T-t)} \bigg] \d x \\
=&\int_0^\infty [x(x-(T-t)L) + x(T-t)L]\exp \bigg[ -\frac{(x-(T-t)L)^2}{8\overline{\sigma}^2_A (T-t)} \bigg] \d x \\
=&\,4\overline{\sigma}^2_A (T-t) \int_0^\infty\exp\bigg[-\frac{(x-(T-t)L)^2}{8\overline{\sigma}^2_A (T-t)}\bigg] \d x
+ 4\overline{\sigma}^2_A (T-t)^ 2L\exp\bigg[-\frac{(T-t)L^2}{8\overline{\sigma}^2_A} \bigg] \\
&\,\, + (T-t)^2L^2 \int_0^\infty\exp\bigg[-\frac{(x-(T-t)L)^2}{8\overline{\sigma}^2_A (T-t)}\bigg] \d x.
\end{align}
Using
\begin{align}
\int_0^\infty \exp \bigg[-\frac{(x-(T-t)L)^2}{8\overline{\sigma}^2_A (T-t)} \bigg] \d x \geq \sqrt{2\pi\overline{\sigma}^2_A (T-t)}, 
\end{align}
the conclusion then follows from Lemma \ref{lem:prop_lower_eigenvalue}.
\end{proof}

\subsection{Ornstein-Uhlenbeck representation} \label{sec:AppOU}

If the reciprocal potential is uniformly convex, the propagation of the generalized convexity profiles is guided by a suitable 1D Ornstein-Uhlenbeck process.

\begin{lemma}[Informal] \label{lem:transitionOU}
If $\kappa_t \equiv \beta \geq 0$, then the solution of \eqref{eq:exp_conv_prof} is given by
\[ f_t (r) = \sqrt{\beta}r-4\sigma^2 \partial_r \log H_t( r), \]
where
\[ H_t(r) 
:= \int_{\bbR} \exp\bigg[ -\frac{1}{4\overline{\sigma}_A^2} \int_0^{|x|} f_T (u) -\sqrt{\beta} u \, \d u \bigg] q_{T-t}(x,r) \d x, \]
using the transition density $q_{t}(x,r)$ of the Ornstein-Uhlenbeck process satisfying 
\[ \d Z_s = -\sqrt{\beta} Z_s \d s + 2 \overline{\sigma}_A \De B_s, \qquad Z_0 = r. \]
\end{lemma}

As for Lemma \ref{lem:time-dependent-prof-reg}, we do not specify the regularity and growth assumptions on $f_T$ at this stage, making the above statement informal.

\begin{proof}
Lemma \ref{lem:time-dependent-prof-reg} here rewrites
\begin{align}
F_t (r) = \int_{\bbR} \bbE \bigg[\exp\bigg[-\frac{1}{4\overline{\sigma}_A^2}\int_t^T \frac{\beta}{2} X_s^2 \d s \bigg] \bigg\vert X_T = x \bigg] p_{T-t}(x,r) \exp \bigg[ -\frac{1}{4\overline{\sigma}_A^2} \int_0^{|x|} f_T(u) \d u \bigg]  \d x,
\end{align}
where $X_s = 2\overline{\sigma}_A B_{s-t} + r$ and $p_t(x,y) :=(8\pi\overline{\sigma}^2_A t)^{-1/2} \exp[-\frac{|x-y|^2}{8\overline{\sigma}^2_A t} ]$.
Up to multiplicative factors, we recognise the Radon-Nikodym derivative of the path-law of $(Z_s)_{t \leq s \leq T}$ w.r.t. the one of $(X_s)_{t \leq s \leq T}$. Indeed, we claim that 
\begin{equation*}\label{eq:Radon-Nikodym_OU}
\bbE \bigg[ \exp\bigg[-\frac{1}{4\overline{\sigma}_A^2}\int_t^T \frac{\beta}{2} X_s^2 \d s \bigg] \bigg\vert X_T = x \bigg] p_{T-t}(x,r) = C q_{T-t}(x,r) \exp \bigg[ \frac{\sqrt{\beta}}{8\overline{\sigma}_A^2}x^2 - \frac{\sqrt{\beta}}{8\overline{\sigma}_A^2}r^2 \bigg],
\end{equation*}
for a constant $C>0$ independent of $r$.
To prove this equality, we denote by $\bbQ$ the path-law of $(Z_s)_{t \leq s \leq T}$, and by $\bbP$ the one of $(X_s)_{t \leq s \leq T}$.
Using Girsanov's theorem, 
\begin{align}
    \frac{\De \bbQ}{\De \bbP}(\omega ) = \exp \bigg[ -\int_t^T\frac{\sqrt{\beta}}{4 \overline\sigma^2_A}\omega_s \De \omega_s -\frac{1}{2}\int_t^T\frac{\beta}{4 \overline\sigma^2_A} \omega_s^2 \De s \bigg],
\end{align}
and Itô's formula then rewrites
\begin{align}
\frac{\De \bbQ}{\De \bbP}( \omega ) = C\exp \bigg[ \frac{\sqrt{\beta}}{8 \overline\sigma^2_A}\omega_t^2 - \frac{\sqrt{\beta}}{8\overline\sigma^2_A}\omega_T^2 -\frac{1}{2}\int_t^T\frac{\beta}{4\sigma^2} \omega_s^2 \d s \bigg],
\end{align}
where $C$ is a constant independent of $\omega$. Conditioning on the endpoints, the desired equality follows.
\end{proof}

\printbibliography

\end{document}